 \documentclass[a4paper,13pt,abstracton]{article}
 \usepackage[utf8]{inputenc}
   \providecommand{\keywords}[1]{\textbf{\textit{Key words:}} #1}
   \usepackage{amsmath,amsthm,verbatim,amssymb,amsfonts,amscd, graphicx,mathrsfs}
 \usepackage{geometry}
 \usepackage{tabularx}
 \usepackage{diagbox}
 \numberwithin{equation}{section}
 \newtheorem{thm}{Theorem}[section]
 \newtheorem{lemma}[thm]{Lemma}
 \newtheorem{definition}[thm]{Definition}
 \newtheorem{cor}[thm]{Corollary}
 \newtheorem{prop}[thm]{Proposition}
 \newtheorem{rem}[thm]{Remark}

 \newtheorem{ex}[thm]{Example}
  \newtheorem{pr}[thm]{Problem}
 \usepackage{enumerate}
 \usepackage[square,comma,super,numbers,sort&compress]{natbib}
  \usepackage{hyperref}
 \hypersetup{colorlinks=true, urlcolor=Cerulean, citecolor=red}
 \usepackage[T1]{fontenc}
 \usepackage{nameref,cleveref}
 \usepackage{nicefrac,xfrac}
 \usepackage[all]{xy}
 \usepackage{color}
\usepackage[section]{placeins}
\newcommand{\tabincell}[2]{\begin{tabular}{@{}#1@{}}#2\end{tabular}}

\begin{document}
\title{\textbf {Vector Bundles on Rational Homogeneous Spaces}}

\author{Rong Du \thanks{School of Mathematical Sciences
Shanghai Key Laboratory of PMMP,
East China Normal University,
Rm. 312, Math. Bldg, No. 500, Dongchuan Road,
Shanghai, 200241, P. R. China,
rdu@math.ecnu.edu.cn.
The Research is Sponsored by National Natural Science Foundation of China (Grant No. 11531007), Natural Science Foundation of China and the Israel Science Foundation (Grant No. 11761141005) and Science and Technology Commission of Shanghai Municipality (Grant No. 18dz2271000).},
Xinyi Fang
\thanks{School of Mathematical Sciences
Shanghai Key Laboratory of PMMP,
East China Normal University,
No. 500, Dongchuan Road,
Shanghai, 200241, P. R. China,
2315885681@qq.com.
The Research is Sponsored by National Natural Science Foundation of China (Grant No. 11531007) and Science and Technology Commission of Shanghai Municipality (Grant No. 18dz2271000).}
and Yun Gao
\thanks{School of Mathematical Sciences,
Shanghai Jiao Tong University,
Shanghai 200240, P. R. of China,
gaoyunmath@sjtu.edu.cn.
The Research is Sponsored by National Natural Science Foundation of China (Grant No. 11531007).
}
}

\date{}
\maketitle

\begin{center}
{Dedicate to people who devote their lives to fight against coronaviruses}
\end{center}

\begin{abstract}
We consider a uniform $r$-bundle $E$ on a complex rational homogeneous space $X$ 
 and show that if $E$ is poly-uniform with respect to all the special families of lines and the rank $r$ is less than or equal to some number that depends only on $X$, then $E$ is either a direct sum of line bundles or $\delta_i$-unstable for some $\delta_i$. So we partially answer a problem posted by Mu\~{n}oz-Occhetta-Sol\'{a} Conde(\cite{M-O-C3}). In particular, if $X$ is a generalized Grassmannian $\mathcal{G}$ and the rank $r$ is less than or equal to some number that depends only on $X$, then $E$ splits as a direct sum of line bundles. We improve the main theorem of Mu\~{n}oz-Occhetta-Sol\'{a} Conde (\cite{M-O-C2} Theorem 3.1) when $X$ is a generalized Grassmannian by considering the Chow ring. Moreover, by calculating the relative tangent bundles between two rational homogeneous spaces, we give explicit bounds for the generalized Grauert-M\"{u}lich-Barth theorem on rational homogeneous spaces.
\end{abstract}

\section{Introduction}
Algebraic vector bundles on a projective variety $X$ over complex number field $\mathbb{C}$ are fundamental research objects in algebraic geometry. However, up to now, algebraic vector bundles are still mysterious for general projective varieties. According to Serre (GAGA), the classification of algebraic vector bundles over $\mathbb{C}$ is equivalent to the classification of holomorphic vector bundles. So there are not only algebraic ways but also analytic ways to handle the problems of vector bundles. For simplicity, we just call vector bundles of rank $r$ or $r$-bundles in the context.

If $X$ is $\mathbb{P}^1$, the structure of a vector bundle on $X$ is quite clear because Grothendieck tells us that it splits as a direct sum of line bundles. However, if $X$ is a projective space and the dimension of it is bigger than or equal to two then the structures of vector bundles on $X$ are not so easy to be determined. Since projective spaces are covered by lines, it is a natural way to consider the restriction of vector bundles to lines. From Grothendieck's result, they split after being restricted to $\mathbb{P}^1$. By semicontinuity theorem, for ``almost all'' lines, $r$-bundle $E$ has the constant splitting type. This means that the lines to which the vector bundle restricts having different splitting type consist a closed subset of a Grassmannian. If the closed subset is empty, such bundles are called uniform vector bundles. Uniform bundles are widely studied not only on projective spaces (\cite{Sch} \cite{Ven} \cite{Sat} \cite{Ele} \cite{E-H-S} \cite{Ell} \cite{Bal}) but also on special Fano manifolds of Picard number one (\cite{Bal2} \cite{K-S} \cite{Guy} \cite{D-F-G} \cite{M-O-C2}). Please see Introduction in \cite{D-F-G} for the details.

Instead of considering vector bundles, Occhetta-Sol\'{a} Conde-Wi\'{s}niewski (\cite{O-C-W}) studied flag bundles which are constructed upon the action of the defining group $G$ on the flag manifold $G/B$. Recently, Mu\~{n}oz-Occhetta-Sol\'{a} Conde (\cite{M-O-C}) studied uniform principle $G$-bundles with $G$ semisimple over Fano manifolds. They present a number of theorems that are flag bundle's versions of some of the central results in the theory of uniform vector bundles. More precisely, they paid special attention to homogeneous filtrations of relative tangent bundles between flag manifolds and generalized the different standard decomposability notions of vector bundles. They used an interesting concept of "tag" of a $G/B$-bundle to describe diagonalizability of any uniform flag bundle of low rank. But they mainly focus on $X$ with Picard number one. In \cite{M-O-C3}, the authors proposed a problem as follows.

\begin{pr}
Classify low rank uniform principle $G$-bundles ($G$ semisimple algebraic group) on rational homogeneous spaces.
\end{pr}

In this paper, we consider uniform bundles on generalized flag varieties or even rational homogeneous spaces with arbitrary Picard numbers and give partial answers to this problem.

Let \[X=G/P\simeq G_1/P_{I_1}\times G_2/P_{I_2}\times \cdots \times G_m/P_{I_m}, \] where $G_i$ is a simple Lie group with \emph{Dynkin diagram} $\mathcal{D}_i$ whose set of nodes is $D_i$ and $P_{I_i}$ is a parabolic subgroup of $G_i$ corresponding to $I_i\subset D_i$. We set $F(I_i):=G_i/P_{I_i}$ by marking on the Dynkin diagram $\mathcal{D}_i$ of $G_i$ the nodes corresponding to $I_i$. Let $\delta_i$ be a node in $\mathcal{D}_i$ and $N(\delta_i)$ be the set of nodes in $\mathcal{D}_i$ that are connected to $\delta_i$.

 If $\delta_i\in I_i$, we call $$\mathcal{M}_i^{\delta_i^c}:=G_i/P_i^{\delta_i^c}\times \widehat{G_i/P_{I_i}}~ (1\le i\le m),$$ the \emph{$i$-th special family} of lines of class $\check{\delta}_i$, where $P_i^{\delta_i^c}:=P_{I_i\backslash \delta_i\cup N(\delta_i)}$ and $\widehat{G_i/P_{I_i}}$ is $G_1/P_{I_1}\times G_2/P_{I_2}\times \cdots \times G_m/P_{I_m}$ by deleting $i$-th term $G_i/P_{I_i}$. Denote by
 $$\mathcal{U}_i^{\delta_i^c}:=G_i/P_{I_i\cup N(\delta_i)}\times \widehat{G_i/P_{I_i}}$$ the \emph{$i$-th universal family} of class $\check{\delta}_i$, which has a natural $\mathbb{P}^1$-bundle structure over $\mathcal{M}_i^{\delta_i^c}$.


We separate our discussion into two cases:

Case I: $N(\delta_i)\subseteq I_i$, then $\mathcal{U}_i^{\delta_i^c}=X$ and we have the natural projection $X\rightarrow \mathcal{M}_i^{\delta_i^c}$;

Case II: $N(\delta_i)\nsubseteq I_i$, then we have the standard diagram
\begin{align}
\xymatrix{
\mathcal{U}_i^{\delta_i^c}\ar[d]^{q_1}   \ar[r]^-{q_2} & \mathcal{M}_i^{\delta_i^c} \\
X.
}
\end{align}


For the definition of "poly-uniform", "$\delta_i$-stable", "$\delta_i$-unstable", "$\varsigma(\mathcal{G})$" and "$\nu(X)$", please see Section 3.

\begin{thm}
On $X$, if an $r$-bundle $E$ is poly-uniform with respect to all the special families of lines and $r\le \nu(X)-2$, then $E$ is $\delta_i$-unstable for some $\delta_i$ ($1\le i \le m$) or $E$ splits as a direct sum of line bundles.
\end{thm}

In particular, if the Picard number of $X$ is one, we improve an interesting theorem of Mu\~{n}oz-Occhetta-Sol\'{a} Conde (\cite{M-O-C2} Theorem 3.1) a little bit (see Table \ref{bound}).

\begin{thm}
Suppose that $E$ is a uniform $r$-bundle on a generalized Grassmann $\mathcal{G}$. If $r\le \varsigma(\mathcal{G})$, then $E$ splits as a direct sum of line bundles.
\end{thm}

By calculating the relative tangent bundles and using Descent Lemma, we can have explicit bounds for the generalized Grauert-M\"{u}lich-Barth theorem on rational homogeneous spaces.





\begin{thm}\label{longroot}
Fix $\delta_i\in I_i$ 
and assume that $\alpha_{\delta_i}$ is not an exposed short root. Let $E$ be a holomorphic $r$-bundle over $X$ of type $\underline{a}_E^{(\delta_i)}=(a_1^{(\delta_i)},\ldots,a_r^{(\delta_i)}), ~ a_1^{(\delta_i)}\geq\cdots\geq a_r^{(\delta_i)}$ with respect to $\mathcal{M}_i^{\delta_i^c}$. If for some $t<r$,
	 \[
	 a_t^{(\delta_i)}-a_{t+1}^{(\delta_i)}\geq
	 \left\{
	 \begin{array}{ll}
	 1, & and~N(\delta_i)~ \text{fits Case I}\\
	 2, & and ~N(\delta_i)~ \text{fits Case II},
	 \end{array}
	 \right.\]
	then there is a normal subsheaf $K\subset E$ of rank $t$ with the following properties: over the open set $V_E=q_1(q_2^{-1}(U_E^{(\delta_i)}))\subset X$, where $U_E^{(\delta_i)}$ is an open set in $\mathcal{M}^{\delta_i^c}$, the sheaf $K$ is a subbundle of $E$, which on the line $L\subset X$ given by $l\in U_E^{(\delta_i)}$ has the form
	$$K|_L\cong\oplus_{s=1}^{t}\mathcal{O}_L(a_s^{(\delta_i)}).$$
\end{thm}

\begin{cor}
With the same assumption as Theorem \ref{longroot}. For a $\delta_i$-semistable $r$-bundle $E$ over $X$ of type $\underline{a}_E^{(\delta_i)}=(a_1^{(\delta_i)},\ldots,a_r^{(\delta_i)}), a_1^{(\delta_i)}\geq\cdots\geq a_r^{(\delta_i)}$ with respect to $\mathcal{M}_i^{\delta_i^c}$, we have \[
a_s^{(\delta_i)}-a_{s+1}^{(\delta_i)}\leq 1~~ \text{for all}~ s=1,\ldots, r-1.
\]
In particular, if $N(\delta_i)$ fits Case I, then we have $a_s^{(\delta_i)}$'s are constant for all $1\leq s\leq r$.
\end{cor}

\begin{thm}\label{shortroot}

Fix $\delta_i\in I_i$ and assume that $\alpha_{\delta_i}$ is an exposed short root. Let $E$ be a holomorphic $r$-bundle over $X$ of type $\underline{a}_E^{(\delta_i)}=(a_1^{(\delta_i)},\ldots,a_r^{(\delta_i)}), ~ a_1^{(\delta_i)}\geq\cdots\geq a_r^{(\delta_i)}$ with respect to $\mathcal{M}_i^{\delta_i^c}$. If for some $t<r$,
	\[
	a_t^{(\delta_i)}-a_{t+1}^{(\delta_i)}\geq
	\left\{
	\begin{array}{ll}
	1, & and~N(\delta_i)~ \text{fits Case I}\\
	4, & and ~N(\delta_i)~ \text{fits Case II},
	\end{array}
	\right.\]
	then there is a normal subsheaf $K\subset E$ of rank $t$ with the following properties: over the open set $V_E=q_1(q_2^{-1}(U_E^{(\delta_i)}))\subset X$, where $U_E^{(\delta_i)}$ is an open set in $\mathcal{M}^{\delta_i^c}$, the sheaf $K$ is a subbundle of $E$, which on the line $L\subset X$ given by $l\in U_E^{(\delta_i)}$ has the form
	$$K|_L\cong\oplus_{s=1}^{t}\mathcal{O}_L(a_s^{(\delta_i)}).$$
\end{thm}

\begin{cor}
With the same assumption as Theorem \ref{shortroot}. For a $\delta_i$-semistable $r$-bundle $E$ over $X$ of type $\underline{a}_E^{(\delta_i)}=(a_1^{(\delta_i)},\ldots,a_r^{(\delta_i)}), a_1^{(\delta_i)}\geq\cdots\geq a_r^{(\delta_i)}$ with respect to $\mathcal{M}_i^{\delta_i^c}$, we have \[
	a_s^{(\delta_i)}-a_{s+1}^{(\delta_i)}\leq 3~~ \text{for all}~ s=1,\ldots, r-1.
	\]
	In particular, if $N(\delta_i)$ fits Case I, then we have $a_s^{(\delta_i)}$'s are constant for all $1\leq s\leq r$.
\end{cor}

If $X=G/B$, where $G$ is a semi-simple Lie group and $B$ is a Borel subgroup of $G$, then we have the following result.
\begin{cor}
If an $r$-bundle $E$ on X is $\delta_i$-semistable for all $i$ and $\delta_i~(1\le i\le m)$, then $E$ splits as a direct sum of line bundles.
\end{cor}
\section{Preliminaries}
Throughout this paper, all algebraic varieties and morphisms will be defined over complex number field $\mathbb{C}$.
\subsection{Semisimple Lie groups and algebras}
In mathematics, Lie group–Lie algebra correspondence allows one to study Lie groups, which are geometric objects, in terms of Lie algebras, which are linear objects. Let $G$ be a semi-simple Lie group. Assume $V$ to be a nonzero finite dimensional complex vector space, fix a maximal torus $H\subset G$, and let $\varphi: G\rightarrow GL(V)$ be a representation of $G$. It is well known that $V$ decomposes into a direct sum of simultaneous eigenspaces
\[V=\bigoplus V_\lambda,\]
where the direct sum run over $\lambda$ in the \emph{character group} of $H$, which is the set of all holomorphic homomorphisms $\lambda$ from $H$ to $\mathbb{C}^*$, and
\[V_\lambda=\{v\in V|\varphi(h)=\lambda(h)v,~\text{for all}~h\in H\}.\]
Since $V$ is a finite dimensional vector space, we have $V_\lambda=0$ for all but finitely many values of $\lambda$. Those values of $\lambda$ for which $V_\lambda\neq 0$ are called the \emph{weights} of $V$, and $V_\lambda$ is called the \emph{weight space}.

In the Lie algebra side,  let $\mathfrak{g}$ be the associated semi-simple Lie algebra of $G$. The maximal torus corresponds to $\mathfrak{h}\subset \mathfrak{g}$ which is an abelian subalgebra of maximal dimension, i.e. Cartan subalgebra. A holomorphic representation $\varphi: G\rightarrow GL(V)$ of a complex Lie group $G$ gives rise to a complex linear representation $\varphi_{\mathfrak{g}}: \mathfrak{g}\rightarrow \mathfrak{gl}(V)$ of the Lie algebra $\mathfrak{g}$ of $G$. Similarly, every finite dimensional representation of $\mathfrak{g}$ admits a decomposition
\[V=\bigoplus_{\lambda\in \{\lambda\in\mathfrak{h}^\vee|V_\lambda\neq 0\}} V_\lambda,\]
where \[V_\lambda=\{v\in V|[h,v]=\lambda(h)v,~\text{for all}~h\in\mathfrak{h}\}.\]
Those $\lambda$ are still called the \emph{weights} of $V$, and $V_\lambda$ is called the \emph{weight space}.
If we apply the above decomposition to $V=\mathfrak{g}$ and $\varphi_\mathfrak{g}$ the adjoint representation, we have \emph{Cartan decomposition} of $\mathfrak{g}$:
\[\mathfrak{g}=\mathfrak{h}\oplus \bigoplus_{\alpha\in \mathfrak{h}^\vee\backslash \{0\}}\mathfrak{g}_\alpha,\]
where \[\mathfrak{g}_\alpha:=\{g\in \mathfrak{g}|[h,g]=\alpha(h)g,~\text{for all}~h\in\mathfrak{h}\}.\]
The elements $\alpha\in \mathfrak{h}^\vee\backslash \{0\}$ for which $\mathfrak{g}_\alpha\neq 0$ are called roots of $\mathfrak{g}$, and the set of these elements will be denoted by $\Phi$ and be called \emph{root system}. For every $\alpha\in \Phi$, $\mathfrak{g}_\alpha$ is called \emph{root space} which is one dimensional.

Fix a linear functional
\[f: \text{span}_{\mathbb{R}}\Phi\rightarrow \mathbb{R}\] whose kernel does not intersect $\Phi$. Let
\[\Phi^+:=\{\alpha\in \Phi| f(\alpha)>0\} ~\text{and}~ \Phi^-:=\{\alpha\in \Phi| f(\alpha)<0\}.\]
Then $\Phi^+$ is called \emph{positive system of roots} and $\Phi^-$ is called \emph{negative system of roots}. Given such a positive system $\Phi^+$, we define the \emph{fundamental system} $\Pi\subset \Phi^+$ as follows: $\alpha\in \Pi$ if and only if $\alpha\in \Phi^+$ and $\alpha$ cannot be expressed as the sum of two elements of $\Phi^+$. A non-zero representation $V$ of $\mathfrak{g}$ is called a \emph{highest weight representation} if it is generated by a vector $v\in V_\lambda$ such that $gv=0$ for all $g\in \bigoplus_{\alpha\in \Phi^+} \mathfrak{g}_\alpha$. In this case, $v$ is called the \emph{highest weight vector}, and $\lambda$ is the highest weight of $V$.

For each $\alpha\in \Phi$ there is a unique element $h_\alpha\in \mathfrak{h}$ such that
\[\alpha(h)=<h_\alpha, h>~\text{for all}~ h\in \mathfrak{h}.\]
The vectors $h_\alpha$ for $\alpha\in\Phi$ span $\mathfrak{h}$. We denote by $\mathfrak{h}_{\mathbb{R}}$ the set of all elements of form $\sum_{i=1}^l a_ih_\alpha$ for $a_i\in \mathbb{R}$.

The \emph{Killing form} $<\alpha, \beta>:=tr(\text{ad}_\alpha\circ \text{ad}_\beta)$ defines a nondegenerated bilinear form on $\mathfrak{h}$, where $\alpha, \beta\in \mathfrak{g}$ and $\text{ad}$ is the adjoint representation. It can be shown that $\mathfrak{h}_{\mathbb{R}}^\vee$ is a Euclidean space with respect to $<, >$. Set $n:=dim_{\mathbb{C}}(\mathfrak{h})$ and $D:=\{1,2,\ldots,n\}$. We identify $D$ with the set of fundamental roots with respect to a choice of maximal torus $T$ and fixed Borel subgroup $B$. It is known that every fundamental system $\Pi\subset \Phi$ can form a basis of $\mathfrak{h}_{\mathbb{R}}^\vee$. Let $\Pi=\{\alpha_1, \ldots , \alpha_n\}$ be a fundamental system. Then we define $A_{ij}$ by
\[A_{ij}=2\frac{<\alpha_i,\alpha_j>}{<\alpha_i,\alpha_i>}\in \mathbb{Z},\quad i,j=1,\ldots n.\]
The matrix $A=(A_{ij})$ is called the \emph{Cartan matrix of $\mathfrak{g}$}.

The \emph{Dynkin diagram of $G$}, which we denoted by $\mathcal{D}:=\mathcal{D}(G)$, is determined by the Cartan matrix. It consists of a graph whose set of nodes is $D$ and where the nodes $i$ and $j$ are joined by $A_{ij}A_{ji}$ edges. When two nodes $i$ and $j$ are joined by a double or triple edge, we add to it an arrow pointing to $i$ if $|A_{ij}|>|A_{ji}|$. We call $\alpha_i$ a short root of $\mathcal{D}$ and $\alpha_j$ a non-short (or long) root of $\mathcal{D}$. (Sometimes, for the sake of narrative convenience, we freely interchange the terminology "node" and "root".) One may prove that there is a one to one correspondence between isomorphism classes of semisimple Lie algebras and Dynkin diagrams of reduced root systems. Moreover, every reduced root system is a disjoint union of mutually orthogonal irreducible root subsystems, each of them corresponding to one of the connected finite Dynkin diagrams $A_n$, $B_n$, $C_n$, $D_n$ $(n\in \mathbb{Z}_{>0})$, $E_6$, $E_7$, $E_8$, $F_4$, $G_2$:

\setlength{\unitlength}{0.4mm}
\begin{center}
\begin{picture}(280,0)(0,120)
\put(10,100){\circle{4}} \put(30,100){\circle{4}}
\put(60,100){\circle{4}} \put(80,100){\circle{4}}
\put(12,100){\line(1,0){16}} \put(40,100){\circle*{2}}
\put(45,100){\circle*{2}} \put(50,100){\circle*{2}}
 \put(62,100){\line(1,0){16}}
 \put(100,100){\circle{4}}
 \put(82,100){\line(1,0){16}}
\put(-10,100){\makebox(0,0)[cc]{$A_n:$}}
\put(7,110){$_1$}
\put(27,110){$_2$}
\put(51,110){$_{n-2}$}
\put(71,110){$_{n-1}$}
\put(97,110){$_{n}$}

\put(210,100){\circle{4}} \put(230,100){\circle{4}}
\put(260,100){\circle{4}} \put(280,100){\circle{4}}
\put(212,100){\line(1,0){16}} \put(240,100){\circle*{2}}
\put(245,100){\circle*{2}} \put(250,100){\circle*{2}}
 \put(262,100){\line(1,0){16}}
 \put(300,100){\circle{4}}
 \put(281,102){\line(1,0){18}}
 \put(281,98){\line(1,0){18}}
 \put(285,103){\line(3,-1){9}}
 \put(285,97){\line(3,1){9}}
\put(190,100){\makebox(0,0)[cc]{$B_n:$}}
\put(207,110){$_1$}
\put(227,110){$_2$}
\put(251,110){$_{n-2}$}
\put(271,110){$_{n-1}$}
\put(297,110){$_{n}$}
 \end{picture}
\end{center}
\vspace{.3cm}

\begin{center}
\begin{picture}(280,20)(0,120)
\put(10,100){\circle{4}} \put(12,100){\line(1,0){16}}
\put(30,100){\circle{4}} \put(40,100){\circle*{2}}
\put(45,100){\circle*{2}} \put(50,100){\circle*{2}}
 \put(60,100){\circle{4}}
\put(62,100){\line(1,0){16}} \put(80,100){\circle{4}}
\put(82,100){\line(3,1){16}} \put(100,94){\circle{4}} \put(82,100){\line(3,-1){16}} \put(100,106){\circle{4}}
\put(-10,100){\makebox(0,0)[cc]{$D_n:$}}
\put(7,110){$_1$}
\put(27,110){$_2$}
\put(51,110){$_{n-3}$}
\put(71,110){$_{n-2}$}
\put(105,106){$_{n-1}$}
\put(105,94){$_{n}$}

\put(210,100){\circle{4}} \put(230,100){\circle{4}}
\put(260,100){\circle{4}} \put(280,100){\circle{4}}
\put(212,100){\line(1,0){16}} \put(240,100){\circle*{2}}
\put(245,100){\circle*{2}} \put(250,100){\circle*{2}}
 \put(262,100){\line(1,0){16}}
 \put(300,100){\circle{4}}
 \put(281,102){\line(1,0){18}}
 \put(281,98){\line(1,0){18}}
 \put(285,100){\line(3,-1){9}}
 \put(285,100){\line(3,1){9}}
\put(190,100){\makebox(0,0)[cc]{$C_n:$}}
\put(207,110){$_1$}
\put(227,110){$_2$}
\put(251,110){$_{n-2}$}
\put(271,110){$_{n-1}$}
\put(297,110){$_{n}$}
 \end{picture}
\end{center}
\vspace{.3cm}

\begin{center}
\begin{picture}(280,20)(0,120)
\put(10,100){\circle{4}} \put(12,100){\line(1,0){16}}
\put(30,100){\circle{4}} \put(32,100){\line(1,0){16}}
\put(50,100){\circle{4}} \put(52,100){\line(1,0){16}}
\put(70,100){\circle{4}} \put(72,100){\line(1,0){16}}
\put(90,100){\circle{4}} \put(50,102){\line(0,1){11}}
\put(50,115){\circle{4}} \put(-10,100){\makebox(0,0)[cc]{$E_6:$}}
\put(7,90){$_1$}
\put(27,90){$_3$}
\put(47,90){$_4$}
\put(67,90){$_5$}
\put(88,90){$_6$}
\put(55,115){$_2$}

\put(210,100){\circle{4}} \put(212,100){\line(1,0){16}}
\put(230,100){\circle{4}} \put(231,102){\line(1,0){18}} \put(231,98){\line(1,0){18}}
\put(250,100){\circle{4}} \put(252,100){\line(1,0){16}}
\put(270,100){\circle{4}}
 \put(190,100){\makebox(0,0)[cc]{$F_4:$}}
 \put(235,103){\line(3,-1){9}}
 \put(235,97){\line(3,1){9}}
\put(207,90){$_1$}
\put(227,90){$_2$}
\put(247,90){$_3$}
\put(267,90){$_4$}
 \end{picture}
\end{center}
\vspace{.3cm}

\begin{center}
\begin{picture}(280,20)(0,120)
\put(10,100){\circle{4}} \put(12,100){\line(1,0){16}}
\put(30,100){\circle{4}} \put(32,100){\line(1,0){16}}
\put(50,100){\circle{4}} \put(52,100){\line(1,0){16}}
\put(70,100){\circle{4}} \put(72,100){\line(1,0){16}}
\put(90,100){\circle{4}} \put(92,100){\line(1,0){16}}
\put(110,100){\circle{4}} \put(70,102){\line(0,1){11}}
\put(70,115){\circle{4}} \put(-10,100){\makebox(0,0)[cc]{$E_7:$}}
\put(7,90){$_1$}
\put(27,90){$_3$}
\put(47,90){$_4$}
\put(67,90){$_5$}
\put(87,90){$_6$}
\put(107,90){$_7$}
\put(74,115){$_2$}

\put(210,100){\circle{4}} \put(230,100){\circle{4}}
\put(211,102){\line(1,0){18}}
\put(212,100){\line(1,0){16}}
\put(211,98){\line(1,0){18}}
\put(190,100){\makebox(0,0)[cc]{$G_2:$}}
\put(207,110){$_1$}
\put(227,110){$_2$}
 \put(215,100){\line(3,-1){9}}
 \put(215,100){\line(3,1){9}}
 \end{picture}
\end{center}
\vspace{.3cm}

\begin{center}
\begin{picture}(280,20)(0,120)
\put(10,100){\circle{4}} \put(12,100){\line(1,0){16}}
\put(30,100){\circle{4}} \put(32,100){\line(1,0){16}}
\put(50,100){\circle{4}} \put(52,100){\line(1,0){16}}
\put(70,100){\circle{4}} \put(72,100){\line(1,0){16}}
\put(90,100){\circle{4}} \put(92,100){\line(1,0){16}}
\put(110,100){\circle{4}} \put(112,100){\line(1,0){16}}
\put(130,100){\circle{4}} \put(90,102){\line(0,1){11}}
\put(90,115){\circle{4}} \put(-10,100){\makebox(0,0)[cc]{$E_8:$}}
\put(7,90){$_1$}
\put(27,90){$_3$}
\put(47,90){$_4$}
\put(67,90){$_5$}
\put(87,90){$_6$}
\put(107,90){$_7$}
\put(127,90){$_8$}
\put(94,115){$_2$}
 \end{picture}
\end{center}
\vspace{2cm}

The connected components of the Dynkin diagram $\mathcal{D}$ determine the simple Lie groups that are factors of the semisimple Lie group $G$, each of them corresponding to one of the Dynkin diagrams above.

\subsection{Parabolic Subgroup and Subalgebra}
A closed subgroup $P$ of $G$ is called \emph{parabolic} if the quotient space $G/P$ is complete, hence projective. A maximal connected solvable subgroup $B$ of $G$ is called a \emph{Borel} subgroup.
We fix a Cartan subalgebra $\mathfrak{h}$. Let
\[\mathfrak{b}=\mathfrak{h}\oplus \bigoplus_{\alpha\in \Phi^-}\mathfrak{g}_\alpha\]
be a fixed Borel subalgebra. It is easy to determine the parabolic subalgebras containing $\mathfrak{b}$. They are all of the form
\[\mathfrak{p}=\mathfrak{b}\oplus\bigoplus_{\alpha\in\Phi^+_P}\mathfrak{g}_\alpha,\]
where $\Phi^+_P$ is a subset of $\Phi^+$ that is closed under the addition of roots. Hence the parabolic subalgebras containing $\mathfrak{b}$ lie in bijection with the subsets of $D$. For some subset $I$ of $D$, we write $\mathfrak{p}_I$ the parabolic subalgebra corresponding to $I$. We define $P_I$ by the parabolic subgroup of $G$ such that its Lie algebra is $\mathfrak{p}_I$. Therefore the parabolic subgroup $P_I$ corresponds to the subset $I$ (so a maximal parabolic subgroup is defined by a single root). Then $G/P_I$ has a minimal homogeneous embedding in projective space of the highest weigh module $V_\lambda$ of $G$ corresponding to the highest weight $\lambda=\sum_{i\in I}\omega_i$, where $\omega_i$ is the $i$-th \emph{fundamental weight} dual to the roots $\alpha_i \in \Pi$, by a very ample line bundle $L_\lambda$.
\subsection{Rational homogeneous spaces}
It is well known that $G/P$ carries a transitive $G$-action, it is a smooth projective variety. Borel and Remmert's classical theorem (\cite{B-R}) states that a projective complex manifold which admits a transitive action of its automorphism group is a direct product of an abelian variety and a rational homogeneous space $G/P$, where $G$ is a semi-simple algebraic group and $P$ is a parabolic subgroup.

Every rational homogeneous space $G/P$ can be decomposed into a product
\[G/P\simeq G_1/P_{I_1}\times G_2/P_{I_2}\times \cdots \times G_m/P_{I_m}\]
of rational homogeneous spaces with simple algebraic group $G_1, \cdots, G_m$. Each rational homogeneous space $G_i/P_{I_i}$, called the \emph{generalized flag manifold}, only depends on the Lie algebra $\mathfrak{g}_i$ of $G_i$, which is classically determined by the marked Dynkin diagram (\cite{L-M}). In the most common notation, we set $F_I:=G/P_I$ by marking on the Dynkin diagram $\mathcal{D}$ of $G$ the nodes corresponding to $I$. For instance, numbering the nodes of $A_n$, the usual flag manifold $F(d_1,\ldots,d_s;n+1)$ corresponds to the marking of $I=\{d_1,\ldots,d_s\}$ (sometimes we omit the braces and just write as $P_{d_1, \ldots, d_s}$).

The two extremal cases correspond to the \emph{generalized complete flag manifolds} (all nodes marked), and the \emph{generalized Grassmannian} (only one node marked).

In \cite{L-M}, the authors explicitly describe the lines through a point of a rational homogeneous space $G/P_I$, where $G$ is a simple Lie group. Let $j \in I$ and $N(j)$ be the set of nodes in $\mathcal{D}$ that are connected to $j$.


\begin{definition}
We call $\alpha_j (j\in I)$ an exposed short root if the connected component of $j$ in $D\backslash (I\backslash j)$ contains root longer than $\alpha_j$, i.e., if an arrow in $D\backslash (I\backslash j)$ points towards $j$.
\end{definition}
\begin{rem}
Obviously, long roots of $\mathcal{D}$ in $I$ are not exposed short roots. If $I$ is a set of single point, i.e.  $X=G/P_I$ is the \emph{generalized Grassmannian}, then the exposed short root is just the usual short root. It's worth mentioning that if $I$ contains all long roots of $\mathcal{D}$, then short roots of $D$ in $I$ are not exposed short roots.
\end{rem}
\begin{thm}(\cite{L-M}~ Theorem 4.3)\label{lines}
Let $I\subseteq D=\{1,\ldots, n\}$. Suppose $G$ to be a simple Lie group. Consider $X=G/P_I$ in its minimal homogeneous embedding. Denote by $F_1(X)$ the space of $\mathbb{P}^{1,}$s in $X$. Then
\begin{enumerate}
\item $F_1(X)=\coprod_{j\in I} F_1^j(X)$, where $F_1^j(X)$ is the space of lines of class $\check{\alpha}_j\in H_2(X, \mathbb{Z})$.
\item If $\alpha_j$ is not an exposed short root, then $F_1^j(X)=G/P_{I\backslash j\cup N(j)}$.
\item If $\alpha_j$ is an exposed short root, then $F_1^j(X)$ is the union of two $G$-orbits, an open orbit and its boundary $G/P_{I\backslash j\cup N(j)}$.
\end{enumerate}
\end{thm}
\begin{rem}
If $I=\{j\}$ and $\alpha_j$ is a long root, then $F_1(X)$ is just the variety of lines on $X$.
\end{rem}
\begin{ex}
Let's consider the Dynkin diagram $A_n$, i.e. $X=SL_{n+1}/P_I$ is the \emph{generalized flag manifold}.

1) For $I=\{k\}$,  $X$ is the usual Grassmannian and $F_1(X)$ is just the variety of lines on $X$.

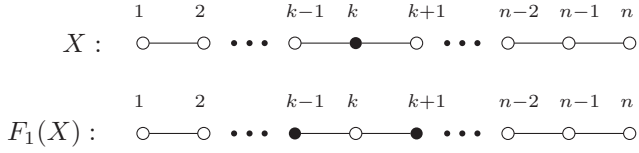
\setlength{\unitlength}{0.4mm}
\begin{center}
	\begin{picture}(280,0)(0,180)
\put(10,170){\circle{4}} \put(12,170){\line(1,0){16}}\put(30,170){\circle{4}}
 \put(40,170){\circle*{2}}
 \put(45,170){\circle*{2}} \put(50,170){\circle*{2}}
\put(60,170){\circle{4}}\put(62,170){\line(1,0){16}} \put(80,170){\circle*{4}}
\put(82,170){\line(1,0){16}}\put(100,170){\circle{4}}
\put(110,170){\circle*{2}}
\put(115,170){\circle*{2}} \put(120,170){\circle*{2}}
\put(130,170){\circle{4}}\put(132,170){\line(1,0){16}} \put(150,170){\circle{4}}
\put(152,170){\line(1,0){16}}\put(170,170){\circle{4}}

\put(-10,170){\makebox(0,0)[cc]{$X:$}}
\put(7,180){$_1$}
\put(27,180){$_2$}
\put(57,180){$_{k-1}$}
\put(77,180){$_{k}$}
\put(97,180){$_{k+1}$}
\put(127,180){$_{n-2}$}
\put(147,180){$_{n-1}$}
\put(167,180){$_{n}$}
\put(10,140){\circle{4}} \put(12,140){\line(1,0){16}}\put(30,140){\circle{4}}
\put(40,140){\circle*{2}}
\put(45,140){\circle*{2}} \put(50,140){\circle*{2}}
\put(60,140){\circle*{4}}\put(62,140){\line(1,0){16}} \put(80,140){\circle{4}}
\put(82,140){\line(1,0){16}}\put(100,140){\circle*{4}}
\put(110,140){\circle*{2}}
\put(115,140){\circle*{2}} \put(120,140){\circle*{2}}
\put(130,140){\circle{4}}\put(132,140){\line(1,0){16}} \put(150,140){\circle{4}}
\put(152,140){\line(1,0){16}}\put(170,140){\circle{4}}
\put(-20,140){\makebox(0,0)[cc]{$F_1(X):$}}
\put(7,150){$_1$}
\put(27,150){$_2$}
\put(57,150){$_{k-1}$}
\put(77,150){$_{k}$}
\put(97,150){$_{k+1}$}
\put(127,150){$_{n-2}$}
\put(147,150){$_{n-1}$}
\put(167,150){$_{n}$}
 \end{picture}
\end{center}
\vspace{2cm}
2) For $I=\{d_1,d_2\}$, $X$ is the usual flag manifold $F(d_1,d_2;n+1)$ and $F_1(X)$ is the disjoint union of $F_1^{d_1}(X)$ and $F_1^{d_2}(X)$.

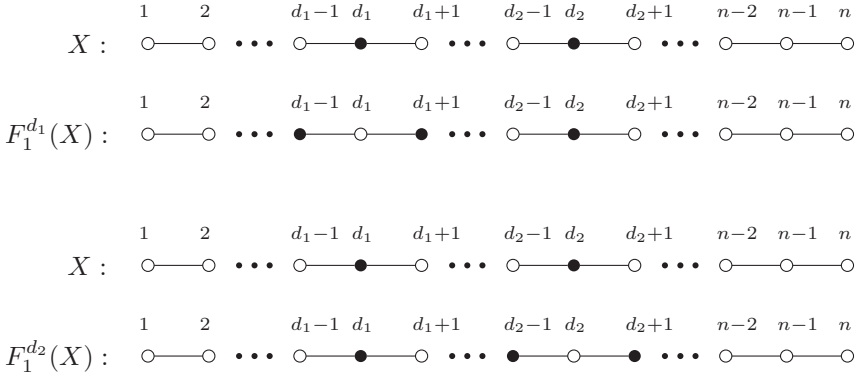
\setlength{\unitlength}{0.4mm}
\begin{center}
	\begin{picture}(280,0)(0,250)
	\put(10,240){\circle{4}} \put(12,240){\line(1,0){16}}\put(30,240){\circle{4}}
	\put(40,240){\circle*{2}}
	\put(45,240){\circle*{2}} \put(50,240){\circle*{2}}
	\put(60,240){\circle{4}}\put(62,240){\line(1,0){16}} \put(80,240){\circle*{4}}
	\put(82,240){\line(1,0){16}}\put(100,240){\circle{4}}
	\put(110,240){\circle*{2}}
	\put(115,240){\circle*{2}} \put(120,240){\circle*{2}}
	\put(130,240){\circle{4}}\put(132,240){\line(1,0){16}} \put(150,240){\circle*{4}}
	\put(152,240){\line(1,0){16}}\put(170,240){\circle{4}}
	\put(180,240){\circle*{2}}
	\put(185,240){\circle*{2}} \put(190,240){\circle*{2}}
	\put(200,240){\circle{4}}\put(202,240){\line(1,0){16}} \put(220,240){\circle{4}}
	\put(222,240){\line(1,0){16}}\put(240,240){\circle{4}}
	\put(-10,240){\makebox(0,0)[cc]{$X:$}}
	\put(7,250){$_1$}
	\put(27,250){$_2$}
	\put(57,250){$_{d_1-1}$}
	\put(77,250){$_{d_1}$}
	\put(97,250){$_{d_1+1}$}
	\put(127,250){$_{d_2-1}$}
	\put(147,250){$_{d_2}$}
	\put(167,250){$_{d_2+1}$}
	\put(197,250){$_{n-2}$}
	\put(217,250){$_{n-1}$}
	\put(237,250){$_{n}$}
\put(10,210){\circle{4}} \put(12,210){\line(1,0){16}}\put(30,210){\circle{4}}
\put(40,210){\circle*{2}}
\put(45,210){\circle*{2}} \put(50,210){\circle*{2}}
\put(60,210){\circle*{4}}\put(62,210){\line(1,0){16}} \put(80,210){\circle{4}}
\put(82,210){\line(1,0){16}}\put(100,210){\circle*{4}}
\put(110,210){\circle*{2}}
\put(115,210){\circle*{2}} \put(120,210){\circle*{2}}
\put(130,210){\circle{4}}\put(132,210){\line(1,0){16}} \put(150,210){\circle*{4}}
\put(152,210){\line(1,0){16}}\put(170,210){\circle{4}}
\put(180,210){\circle*{2}}
\put(185,210){\circle*{2}} \put(190,210){\circle*{2}}
\put(200,210){\circle{4}}\put(202,210){\line(1,0){16}} \put(220,210){\circle{4}}
\put(222,210){\line(1,0){16}}\put(240,210){\circle{4}}
\put(-20,210){\makebox(0,0)[cc]{$F_1^{d_1}(X):$}}
\put(7,220){$_1$}
\put(27,220){$_2$}
\put(57,220){$_{d_1-1}$}
\put(77,220){$_{d_1}$}
\put(97,220){$_{d_1+1}$}
\put(127,220){$_{d_2-1}$}
\put(147,220){$_{d_2}$}
\put(167,220){$_{d_2+1}$}
\put(197,220){$_{n-2}$}
\put(217,220){$_{n-1}$}
\put(237,220){$_{n}$}
	\end{picture}
\end{center}
\vspace{2cm}
\setlength{\unitlength}{0.4mm}
\begin{center}
	\begin{picture}(280,0)(0,250)
	\put(10,240){\circle{4}} \put(12,240){\line(1,0){16}}\put(30,240){\circle{4}}
	\put(40,240){\circle*{2}}
	\put(45,240){\circle*{2}} \put(50,240){\circle*{2}}
	\put(60,240){\circle{4}}\put(62,240){\line(1,0){16}} \put(80,240){\circle*{4}}
	\put(82,240){\line(1,0){16}}\put(100,240){\circle{4}}
	\put(110,240){\circle*{2}}
	\put(115,240){\circle*{2}} \put(120,240){\circle*{2}}
	\put(130,240){\circle{4}}\put(132,240){\line(1,0){16}} \put(150,240){\circle*{4}}
	\put(152,240){\line(1,0){16}}\put(170,240){\circle{4}}
	\put(180,240){\circle*{2}}
	\put(185,240){\circle*{2}} \put(190,240){\circle*{2}}
	\put(200,240){\circle{4}}\put(202,240){\line(1,0){16}} \put(220,240){\circle{4}}
	\put(222,240){\line(1,0){16}}\put(240,240){\circle{4}}
	\put(-10,240){\makebox(0,0)[cc]{$X:$}}
	\put(7,250){$_1$}
	\put(27,250){$_2$}
	\put(57,250){$_{d_1-1}$}
	\put(77,250){$_{d_1}$}
	\put(97,250){$_{d_1+1}$}
	\put(127,250){$_{d_2-1}$}
	\put(147,250){$_{d_2}$}
	\put(167,250){$_{d_2+1}$}
	\put(197,250){$_{n-2}$}
	\put(217,250){$_{n-1}$}
	\put(237,250){$_{n}$}
	\put(10,210){\circle{4}} \put(12,210){\line(1,0){16}}\put(30,210){\circle{4}}
	\put(40,210){\circle*{2}}
	\put(45,210){\circle*{2}} \put(50,210){\circle*{2}}
	\put(60,210){\circle{4}}\put(62,210){\line(1,0){16}} \put(80,210){\circle*{4}}
	\put(82,210){\line(1,0){16}}\put(100,210){\circle{4}}
	\put(110,210){\circle*{2}}
	\put(115,210){\circle*{2}} \put(120,210){\circle*{2}}
	\put(130,210){\circle*{4}}\put(132,210){\line(1,0){16}} \put(150,210){\circle{4}}
	\put(152,210){\line(1,0){16}}\put(170,210){\circle*{4}}
	\put(180,210){\circle*{2}}
	\put(185,210){\circle*{2}} \put(190,210){\circle*{2}}
	\put(200,210){\circle{4}}\put(202,210){\line(1,0){16}} \put(220,210){\circle{4}}
	\put(222,210){\line(1,0){16}}\put(240,210){\circle{4}}
	\put(-20,210){\makebox(0,0)[cc]{$F_1^{d_2}(X):$}}
	\put(7,220){$_1$}
	\put(27,220){$_2$}
	\put(57,220){$_{d_1-1}$}
	\put(77,220){$_{d_1}$}
	\put(97,220){$_{d_1+1}$}
	\put(127,220){$_{d_2-1}$}
	\put(147,220){$_{d_2}$}
	\put(167,220){$_{d_2+1}$}
	\put(197,220){$_{n-2}$}
	\put(217,220){$_{n-1}$}
	\put(237,220){$_{n}$}
	\end{picture}
\end{center}
\vspace{2cm}
\end{ex}	
There is a similar statement for $C_x\subseteq \mathbb{P}T_xX$, the set of tangent directions to lines on $X$ passing through a fixed point $x$. It is a disjoint  union of spaces of lines of class $\check{\alpha}$ through $x$.
\begin{thm}(\cite{L-M}~ Theorem 4.8)\label{lines2}
	Let $I\subseteq D=\{1,\ldots, n\}$ and $j\in I$. Suppose $G$ to be a simple Lie group. Consider $X=G/P_I$ in its minimal homogeneous embedding. Let $H$ be the semisimple part of $P_I$ and $D(H)$ be the components of $(\overline{D\backslash I})\backslash j$ containing an element of $N(j)$, where $\overline{D\backslash I}$ means $D\backslash I$ plus any nodes of $I$ attached to a node of $D\backslash I$. Denote by $C_x^j$ the space of lines of class $\check{\alpha}_j$ through $x$. Then
	\begin{enumerate}
		\item If $\alpha_j$ is not an exposed short root, then $C_x^j=H/P_{N(j)}$.
		\item If $\alpha_j$ is an exposed short root, then $C_x^j$ is a union of an open $P_I$-orbit and its boundary $H/P_{N(j)}$.
	\end{enumerate}
\end{thm}
\begin{rem}\label{vmrt}
If $I=\{j\}$, then the set of nodes of the Dynkin diagram $H$ is $D(H)=D\backslash {j}$. $P_{N(j)}$ is a parabolic subgroup of $H$ by marking in $D(H)$ the nodes in $D$ that are connected to $j$ and $C_x^j=H/P_{N(j)}$. Moreover, if $\alpha_j$ is a long root then $C_x^j$ is just the variety of lines through fixed points, i.e. so-called $\text{VMRT}$s. We refer to \cite{Hwa} for a complete account on $\text{VMRT}$s.
\end{rem}
\begin{ex}
	Let's consider the Dynkin diagram $A_n$, i.e. $X=SL_{n+1}/P_I$ is the \emph{generalized flag manifold}.
	
	1) For $I=\{k\}$,  $X$ is the usual Grassmannian and $C_x=\mathbb{P}^{k-1}\times\mathbb{P}^{n-k}$ is just the variety of lines through $x$.
	\setlength{\unitlength}{0.4mm}
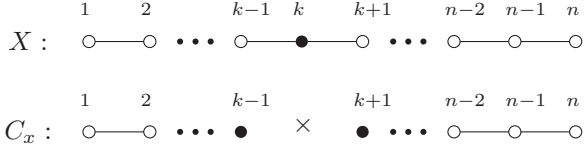
	\begin{center}
		\begin{picture}(280,0)(0,180)
		\put(10,170){\circle{4}} \put(12,170){\line(1,0){16}}\put(30,170){\circle{4}}
		\put(40,170){\circle*{2}}
		\put(45,170){\circle*{2}} \put(50,170){\circle*{2}}
		\put(60,170){\circle{4}}\put(62,170){\line(1,0){16}} \put(80,170){\circle*{4}}
		\put(82,170){\line(1,0){16}}\put(100,170){\circle{4}}
		\put(110,170){\circle*{2}}
		\put(115,170){\circle*{2}} \put(120,170){\circle*{2}}
		\put(130,170){\circle{4}}\put(132,170){\line(1,0){16}} \put(150,170){\circle{4}}
		\put(152,170){\line(1,0){16}}\put(170,170){\circle{4}}
		
		\put(-10,170){\makebox(0,0)[cc]{$X:$}}
		\put(7,180){$_1$}
		\put(27,180){$_2$}
		\put(57,180){$_{k-1}$}
		\put(77,180){$_{k}$}
		\put(97,180){$_{k+1}$}
		\put(127,180){$_{n-2}$}
		\put(147,180){$_{n-1}$}
		\put(167,180){$_{n}$}
		\put(10,140){\circle{4}} \put(12,140){\line(1,0){16}}\put(30,140){\circle{4}}
		\put(40,140){\circle*{2}}
		\put(45,140){\circle*{2}} \put(50,140){\circle*{2}}
		\put(60,140){\circle*{4}}\put(77,140){$\times$}
		\put(100,140){\circle*{4}}
		\put(110,140){\circle*{2}}
		\put(115,140){\circle*{2}} \put(120,140){\circle*{2}}
		\put(130,140){\circle{4}}\put(132,140){\line(1,0){16}} \put(150,140){\circle{4}}
		\put(152,140){\line(1,0){16}}\put(170,140){\circle{4}}
		
		\put(-10,140){\makebox(0,0)[cc]{$C_x:$}}
		\put(7,150){$_1$}
		\put(27,150){$_2$}
		\put(57,150){$_{k-1}$}
		\put(97,150){$_{k+1}$}
		\put(127,150){$_{n-2}$}
		\put(147,150){$_{n-1}$}
		\put(167,150){$_{n}$}
		\end{picture}
	\end{center}
	\vspace{2cm}
	2) For $I=\{d_1,d_2\}$, $X$ is the usual flag manifold $F(d_1,d_2;n+1)$ and $C_x$ is the disjoint union of $C_x^{d_1}$ and $C_x^{d_2}$.
	\[
	C_x^{d_1}=\mathbb{P}^{d_1-1}\times\mathbb{P}^{d_2-d_1-1},~~C_x^{d_2}=\mathbb{P}^{d_2-d_1-1}\times\mathbb{P}^{n-d_2}.
	\]
	
	\setlength{\unitlength}{0.4mm}
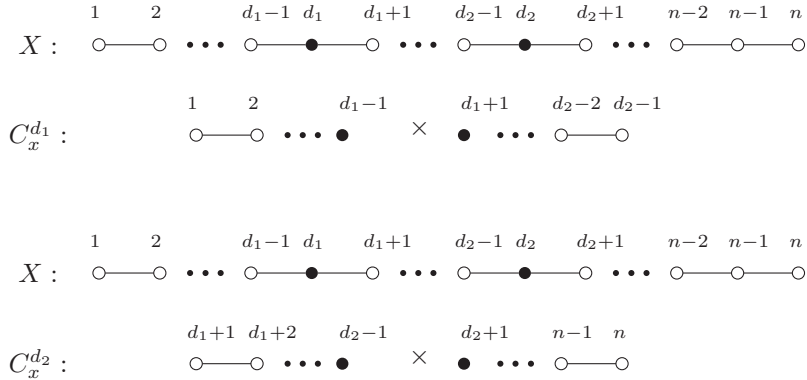
	\begin{center}
		\begin{picture}(280,0)(0,250)
		\put(10,240){\circle{4}} \put(12,240){\line(1,0){16}}\put(30,240){\circle{4}}
		\put(40,240){\circle*{2}}
		\put(45,240){\circle*{2}} \put(50,240){\circle*{2}}
		\put(60,240){\circle{4}}\put(62,240){\line(1,0){16}} \put(80,240){\circle*{4}}
		\put(82,240){\line(1,0){16}}\put(100,240){\circle{4}}
		\put(110,240){\circle*{2}}
		\put(115,240){\circle*{2}} \put(120,240){\circle*{2}}
		\put(130,240){\circle{4}}\put(132,240){\line(1,0){16}} \put(150,240){\circle*{4}}
		\put(152,240){\line(1,0){16}}\put(170,240){\circle{4}}
		\put(180,240){\circle*{2}}
		\put(185,240){\circle*{2}} \put(190,240){\circle*{2}}
		\put(200,240){\circle{4}}\put(202,240){\line(1,0){16}} \put(220,240){\circle{4}}
		\put(222,240){\line(1,0){16}}\put(240,240){\circle{4}}
		\put(-10,240){\makebox(0,0)[cc]{$X:$}}
		\put(7,250){$_1$}
		\put(27,250){$_2$}
		\put(57,250){$_{d_1-1}$}
		\put(77,250){$_{d_1}$}
		\put(97,250){$_{d_1+1}$}
		\put(127,250){$_{d_2-1}$}
		\put(147,250){$_{d_2}$}
		\put(167,250){$_{d_2+1}$}
		\put(197,250){$_{n-2}$}
		\put(217,250){$_{n-1}$}
		\put(237,250){$_{n}$}
		\put(42,210){\circle{4}} \put(44,210){\line(1,0){16}}\put(62,210){\circle{4}}
		\put(72,210){\circle*{2}}
		\put(77,210){\circle*{2}} \put(82,210){\circle*{2}}
		\put(90,210){\circle*{4}}\put(112,210){$\times$}
		\put(130,210){\circle*{4}}
		\put(142,210){\circle*{2}}
		\put(147,210){\circle*{2}} \put(152,210){\circle*{2}}
		\put(162,210){\circle{4}}\put(164,210){\line(1,0){16}} \put(182,210){\circle{4}}
		\put(-10,210){\makebox(0,0)[cc]{$C_x^{d_1}:$}}
		\put(39,220){$_1$}
		\put(59,220){$_2$}
		\put(89,220){$_{d_1-1}$}
		\put(129,220){$_{d_1+1}$}
		\put(159,220){$_{d_2-2}$}
		\put(179,220){$_{d_2-1}$}
		\end{picture}
	\end{center}
	\vspace{2cm}
		\setlength{\unitlength}{0.4mm}
		\begin{center}
			\begin{picture}(280,0)(0,250)
			\put(10,240){\circle{4}} \put(12,240){\line(1,0){16}}\put(30,240){\circle{4}}
			\put(40,240){\circle*{2}}
			\put(45,240){\circle*{2}} \put(50,240){\circle*{2}}
			\put(60,240){\circle{4}}\put(62,240){\line(1,0){16}} \put(80,240){\circle*{4}}
			\put(82,240){\line(1,0){16}}\put(100,240){\circle{4}}
			\put(110,240){\circle*{2}}
			\put(115,240){\circle*{2}} \put(120,240){\circle*{2}}
			\put(130,240){\circle{4}}\put(132,240){\line(1,0){16}} \put(150,240){\circle*{4}}
			\put(152,240){\line(1,0){16}}\put(170,240){\circle{4}}
			\put(180,240){\circle*{2}}
			\put(185,240){\circle*{2}} \put(190,240){\circle*{2}}
			\put(200,240){\circle{4}}\put(202,240){\line(1,0){16}} \put(220,240){\circle{4}}
			\put(222,240){\line(1,0){16}}\put(240,240){\circle{4}}
			\put(-10,240){\makebox(0,0)[cc]{$X:$}}
			\put(7,250){$_1$}
			\put(27,250){$_2$}
			\put(57,250){$_{d_1-1}$}
			\put(77,250){$_{d_1}$}
			\put(97,250){$_{d_1+1}$}
			\put(127,250){$_{d_2-1}$}
			\put(147,250){$_{d_2}$}
			\put(167,250){$_{d_2+1}$}
			\put(197,250){$_{n-2}$}
			\put(217,250){$_{n-1}$}
			\put(237,250){$_{n}$}
			\put(42,210){\circle{4}} \put(44,210){\line(1,0){16}}\put(62,210){\circle{4}}
			\put(72,210){\circle*{2}}
			\put(77,210){\circle*{2}} \put(82,210){\circle*{2}}
			\put(90,210){\circle*{4}}\put(112,210){$\times$}
			\put(130,210){\circle*{4}}
			\put(142,210){\circle*{2}}
			\put(147,210){\circle*{2}} \put(152,210){\circle*{2}}
			\put(162,210){\circle{4}}\put(164,210){\line(1,0){16}} \put(182,210){\circle{4}}
			\put(-10,210){\makebox(0,0)[cc]{$C_x^{d_2}:$}}
			\put(39,220){$_{d_1+1}$}
			\put(59,220){$_{d_1+2}$}
			\put(89,220){$_{d_2-1}$}
			\put(129,220){$_{d_2+1}$}
			\put(159,220){$_{n-1}$}
			\put(179,220){$_{n}$}
			\end{picture}
		\end{center}
		\vspace{2cm}
\end{ex}	

Not only $\mathbb{P}^1$ but also all linear spaces can be read from the marked Dynkin diagrams.

\begin{thm}(\cite{L-M}~ Theorem 4.9, 4.14)\label{hdp}
Let $G$ be a simple group and $X=G/P_S$ is a rational homogeneous space. Let $F_{k}^{\alpha}(X)$ denote the variety parameterizing the $\alpha$-class $\mathbb{P}^k$'s on $X$.
\begin{itemize}
  \item If $\alpha\in S$ is not an exposed short root, then for all $k$, $F_{k}^{\alpha}(X)$ is the disjoint union of homogeneous spaces $G/P_{\sum\beta_j}$, where $\{\beta_j\}$ is a set of positive roots such that the component of $\mathcal{D}\backslash \{\beta_j\}$ containing $\alpha$ is isomorphic to $\mathcal{D}(A_k)$, intersects $S$ only in $\alpha$, and $\alpha$ is an extremal node of  this component.
  \item If $\alpha\in S$ is an exposed short root, then for all $k$, $F_{k}^{\alpha}(X)$ consists of a finite number of $G$-orbits.
\end{itemize}
\end{thm}
\section{Uniform vector bundles}
Given a smooth projective variety $X$ and a vector bundle $E$ on $X$, we denote $\mathcal{M}$ to be an unsplit family of rational curves on $X$. $\mathcal{M}$ is called \emph{unsplit} if $\mathcal{M}$ is a proper $\mathbb{C}$-scheme. We say that $E$ is \emph{uniform} with respect to $\mathcal{M}$ if the restriction of $E$ to the normalization of every curve in $\mathcal{M}$ splits as a direct sum of line bundles with the same splitting type. If $X$ is a generalized Grassmannian $G/P_k$, then we just call $E$ uniform without mention the unspilt family $\mathcal{M}$.

\subsection{Uniform vector bundles on generalized Grassmannians}
Along this section we will work on uniform vector bundles on rational homogeneous spaces of Picard number one, i.e. generalized Grassmannians. Let $G$ be a simple Lie group and $D=\{1,2,\ldots,n\}$ be the set of nodes of the Dynkin diagram $\mathcal{D}$ of $G$. Denote by $P_k$ the parabolic subgroup of $G$ corresponds to the node $k$. Consider the generalized Grassmannian $\mathcal{G}=G/P_k$ or, for brevity, $\mathcal{D}/P_k$. Denote by $\mathcal{M}:=G/P_{N(k)}$ the generalized flag manifold defined by the marked Dynkin diagram $(\mathcal{D},N(k))$ and by $\mathcal{U}:=G/P_{k,N(k)}=G/(P_{N(k)}\cap P_k)$ the universal family, which has a natural $\mathbb{P}^{1}$-bundle structure over $\mathcal{M}$, i.e. we have the natural diagram
\begin{align}
	\xymatrix{
		\mathcal{U}\ar[d]^{p}   \ar[r]^-{q} & \mathcal{M}\\
		\mathcal{G}=G/P_k.
	}
\end{align}	
Remarkably, $\mathcal{M}$ defined above is indeed an unsplit family of rational curves on $\mathcal{G}$. Given $x\in \mathcal{G}$, $\mathcal{M}_x=q(p^{-1}(x))$, which we call \emph{the special family} of lines of class $\check{\alpha_k}$ through $x$, coincides with $H/P_{N(k)}$ by Remark \ref{vmrt}, where the set of nodes of the Dynkin diagram $H$ is $D(H)=D\backslash {k}$.

When $k$ is an extremal node, that is, the subdiagram $\mathcal{D}(H)$ is connected. Remarkably, in the case $\mathcal{D}=D_n$, i.e. $G=SO(2n)$, since $G/P_{n-1}\cong G/P_n$, we only need to think about the extremal node $n$. Similarly, since $E_6/P_1\cong E_6/P_6$, we just consider the the extremal node $1$ in $E_6$. According to Theorem \ref{lines2}, $\mathcal{M}_x$ has the following possibilities:
\begin{itemize}
	\item Projective spaces or smooth quadrics,
	\item Grassmannians,
	\item Spinor varities,
	\item $E_6/P_6$, $E_7/P_7$, $C_3/P_3$.
\end{itemize}
The possibilities are list in Table \ref{t1} below.

\begin{table}[htbp]
	\caption{$\mathcal{M}_x$ corresponding to an extremal node} \label{t1}
	\centering
	\begin{tabular}{|c|c|c|c|c|c|c|c|}
		\hline
		\diagbox{$node$}{$\mathcal{M}_x$}{$\mathcal{D}$}& $A_n$&$B_n$&$C_n$&$D_n$&$E_n(n=6,7,8)$&$F_n(n=4)$&$G_n(n=2)$\\
		\hline
		$1$&$\mathbb{P}_{n-1}$ & $Q_{2n-3}$ & $\mathbb{P}_{2n-3}$&$Q_{2n-4}$&$\mathcal{S}_{n-2}$&$C_3/P_3$&$\mathbb{P}_1$ \\
		\hline
		$n$&$\mathbb{P}_{n-1}$ & $\mathbb{P}_{n-1}$ & $\mathbb{P}_{n-1}$&$G(2,n)$&$E_{n-1}/P_{n-1}(n\neq 6)$&$\mathcal{S}_3$&$\mathbb{P}_1$ \\
		\hline
		$2$& &  & &&$G(3,n)$&&\\
		\hline
	\end{tabular}
\end{table}

We observe that for $x\in \mathcal{G}$, the morphism from $\mathcal{M}_x$ to Grassmannian plays a critical role in determining whether a uniform vector bundle can split as a direct sum of line bundles. Let $\varsigma$ be a positive integer smaller than or equal to $\text{dim}~\mathcal{M}_x$. As long as we show that the morphism $\mathcal{M}_x\rightarrow G(t,\varsigma)$ can only be constant for any integer $1\le t\le [\frac{\varsigma}{2}]$ and every $x\in \mathcal{G}$, then every uniform $r$-bundle on $\mathcal{G}$ splits for $r\le \varsigma$. We suggest that interested readers refer to Theorem 3.1 in paper \cite{M-O-C2} for details. Now, let's analyze the morphism $\mathcal{M}_x\rightarrow G(t,\varsigma)$ one by one according to the probabilities of $\mathcal{M}_x$.
\\
\\
Case \uppercase\expandafter{\romannumeral 1}. When $\mathcal{M}_x$ is a projective space $\mathbb{P}^N$ or a smooth quadric $Q^N~(N=2m+1)$, then their Chow rings have the form
$$\mathbb{Z}[\mathcal{H}]/(\mathcal{H}^{N+1}),$$
where $\mathcal{H}$ is a hyperplane section. In particular, $\text{dim} H^{2t}(\mathcal{M}_x,\mathbb{C})=1$ for every $t\le [\frac{N}{2}]$. By the proof of Lemma 3.4 in paper \cite{M-O-C2}, the only morphisms $\mathcal{M}_x\rightarrow G(t,N)$ are constant for any integer $1\le t\le [\frac{N}{2}]$.

When $\mathcal{M}_x$ is a smooth quadric $Q^N~(N=2m)$, since
$$A(Q^{2m})=\mathbb{Z}[\mathcal{H},\mathcal{U}]/(\mathcal{H}^{2m+1}, 2\mathcal{H}\mathcal{U}-\mathcal{H}^{m+1},\mathcal{H}^m\mathcal{U}-\mathcal{U}^2),$$
where $\mathcal{H}$ is a hyperplane section and $\mathcal{U}$ is a subvariety of codimension $m$, then we get that $\mathcal{M}_x\rightarrow G(t,N-1)$  can only be constant map similarly.
\\
\\
Case \uppercase\expandafter{\romannumeral 2}. $\mathcal{M}_x$ is Grassmannian $G(d,n)~(2\le d\le n-d)$. We claim that the only morphisms $G(d,n)\rightarrow G(t,n-d+1)$ are constant for any integer $1\le t\le [\frac{n-d+1}{2}]$.
\begin{lemma}
There are no nonconstant maps from $G(d,n)~(2\le d\le n-d)$ to $G(t,n-d+1)$ for any integer $1\le t\le [\frac{n-d+1}{2}]$.
\end{lemma}
\begin{proof}
Assume that we have a nonconstant morphism $\phi : =G(d,n)\rightarrow G(t,n-d+1)$. Then there exists a maximal linear subspace $\mathbb{P}_{n-d}$ such that $\phi$ restricts to it is also nonconstant. Denote by $\psi$ to be the restriction map. Let's consider $\psi^{\ast}H_t,\psi^{\ast}Q_{n-d+1-t}$ and $\phi^{\ast}H_t,\phi^{\ast}Q_{n-d+1-t}$(the pull back of universal bundle $H_t$ and universal  quotient bundle $Q_{n-d+1-t}$ under $\psi,\phi$). Denote by $c_1,\ldots,c_t$ and $d_1,\ldots,d_{n-d+1-t}$ the Chern classes of $\psi^{\ast}H_t$ and $\psi^{\ast}Q_{n-d+1-t}$, respectively, and by $C_1,\ldots,C_t$ and $D_1,\ldots,D_{n-d+1-t}$ the Chern classes of $\phi^{\ast}H_t$ and $\phi^{\ast}Q_{n-d+1-t}$, respectively.

On $G(d,n)~(2\le d\le n-d)$, we have an exact sequence\[
0\rightarrow\phi^{\ast}H_t\rightarrow\mathcal{O}^{\oplus n-d+1}_{G(d,n)}\rightarrow\phi^{\ast}Q_{n-d+1-t}\rightarrow 0
\]
which is the pull back of the universal exact sequence \[0\rightarrow H_t\rightarrow\mathcal{O}^{\oplus n-d+1}_{G(t,n-d+1-t)}\rightarrow Q_{n-d+1-t}\rightarrow 0\] on $G(t,n-d+1)$.
Then \[
c(\phi^{\ast}H_t)\cdot c(\phi^{\ast}Q_{n-d+1-t})=1,
\]
i.e. \[(1+C_1+\cdots+C_t)\cdot(1+D_1+\cdots+D_{n-d+1-t})=1.\]
 Since $n-d+1<d(n-d)=\text{dim} G(d,n)$, obviously we have $C_t\cdot D_{n-d+1-t}=0$.

On the other hand, on $\mathbb{P}_{n-d}\subset G(d,n)$, we also have an exact sequence\[
0\rightarrow\psi^{\ast}H_t\rightarrow\mathcal{O}^{\oplus n-d+1}_{\mathbb{P}_{n-d}}\rightarrow\psi^{\ast}Q_{n-d+1-t}\rightarrow 0.
\]
Then \[
c(\psi^{\ast}H_t)\cdot c(\psi^{\ast}Q_{n-d+1-t})=1,
\]
i.e. \[(1+c_1+\cdots+c_t)\cdot(1+d_1+\cdots+d_{n-d+1-t})=1.\] Combining the above equation and the $\psi$ nonconstant assumption, i.e. \[c_1=\text{deg} \psi^{\ast}\mathcal{O}_{G(t,n-d+1)}(1)\neq 0,\] we get $c_t\neq 0$ and $d_{n-d+1-t}\neq 0$.

In order to show that this contradicts that $C_t\cdot D_{n-d+1-t}=0$, we need to use some Schubert classes in $G(d,n)$ (see \cite{Da} Chapter 4). Let's review some basic facts and fix some notation first. Choose a complete flag $\mathcal{V}$ in $\mathbb{C}^{n}$, that is, a nested sequence of subspaces
$$0\subset V_1\subset\cdots\subset V_{n-d}\subset V_{n-d}=\mathbb{C}^{n}$$
with dim $V_i=i$. For a sequence $a=(a_1,\ldots,a_d)$ with $n-d\ge a_1\ge\ldots\ge a_d\ge 0$, we define the Schubert cycle $\Sigma_a(\mathcal{V})\subset G(d,n)$ to be the closed subset
$$ \Sigma_a(\mathcal{V})=\{\Lambda\in G(d,n)|\text{dim}(V_{n-d+i-a_i}\cap \Lambda\ge i)~ \text{for all}~ i\}.$$
The class $\sigma_a:=[\Sigma_a(\mathcal{V})]$ ia called Schubert class. By Theorem 4.1 in \cite{Da}, $\sigma_a$ is of codimension $a_1+\cdots+a_d$. To simplify notation, we generally suppress trailing zeros in the indices and write $\sigma_{a_1,\ldots,a_s}$ in place of $\sigma_{a_1,\ldots,a_s,0,\ldots,0}$. With this notation, $\mathbb{P}_{n-d}$ can be represented as $\sigma_{n-d,\ldots,n-d}$ and we may write
$$(-1)^{t}C_t=\sum_{a_1+\cdots+a_d=t}x_{a_1,\ldots,a_d}\sigma_{a_1,\ldots,a_d} ~ \text{and}~ D_{n-d+1-t}=\sum_{a_1+\cdots+a_d=n-d+1-t}y_{a_1,\ldots,a_d}\sigma_{a_1,\ldots,a_d},$$
where the ${x_{a_1,\ldots,a_d}}'$s and ${y_{a_1,\ldots,a_d}}'$s are non negative integers by the nefness of $H_t^{\vee}$ (the dual of $H_t$) and $Q_{n-d+1-t}$.

Pieri's formula (see Proposition 4.9 in \cite{Da}) tells us that for any integers $b_1,\ldots,b_d$ with $b_1+\cdots+b_d=l\le n-d$,
$$\sigma_{b_1}\cdots\sigma_{b_d}\cap \mathbb{P}_{n-d}=\sigma_{b_1}\cdots\sigma_{b_d}\cdot\sigma_{n-d,\ldots,n-d,0}=\sigma_{n-d,\ldots,n-d,l}.$$
Using Giambelli's formula (see Proposition 4.16 in \cite{Da}) for $\sigma_{a_1,\ldots,a_d}$'s, expanding the determinants and then intersecting with $\sigma_{n-d,\ldots,n-d,0}$, we can immediately get the following identities:
\begin{align*}
c_t&=C_t\cap\mathbb{P}_{n-d}\\
&=(-1)^tx_{t,0,\ldots,0}\sigma_t\cdot\sigma_{n-d,\ldots,n-d,0}\\
&=(-1)^tx_{t,0,\ldots,0}\sigma_{n-d,\ldots,n-d,t};\\
d_{n-d+1-t}&=D_{n-d+1-t}\cap\mathbb{P}_{n-d}\\
&=y_{n-d+1-t,0,\ldots,0}\sigma_{n-d+1-t}\cdot\sigma_{n-d,\ldots,n-d,0}\\
&=y_{n-d+1-t,0,\ldots,0}\sigma_{n-d,\ldots,n-d,n-d+1-t}.
\end{align*}
Hence $x_{t,0,\ldots,0}y_{n-d+1-t,0,\ldots,0}\neq 0$. But
$$0=C_t\cdot D_{n-d+1-t}=x_{t,0,\ldots,0}y_{n-d+1-t,0,\ldots,0}\sigma_t\cdot \sigma_{n-d+1-t}+\ldots,$$
where the summation is linear combination of Schubert cycles with non-negative coefficients by the Littlewood-Richardson formula. Therefore $$x_{t,0,\ldots,0}y_{n-d+1-t,0,\ldots,0}=0,$$ a contradiction.
\end{proof}
When $\mathcal{G}=D_n/P_n$, $M_x=G(2,n)$. By the above Lemma, we obtain the only morphisms $G(2,n)\rightarrow G(t,n-1)$ are constant for any integer $1\le t\le [\frac{n-1}{2}]$.

When $\mathcal{G}=E_n/P_2~(n=6,7,8)$, $M_x=G(3,n)~(n=6,7,8)$. By the above Lemma, we obtain the only morphisms $G(3,n)\rightarrow G(t,n-2)~(n=6,7,8)$ are constant for any integer $1\le t\le [\frac{n-2}{2}]$. Remarkably, in these cases, the value of $\varsigma$ can be appropriately enlarged. Since $G(3,n)~(n=6,7,8)$ is $3(n-3)$ dimensional, we can easily know
$$\text{dim}(G(3,6))=9>\text{dim}G(t,5),$$
$$\text{dim}(G(3,7))=12>\text{dim}G(t,6) $$ and
$$\text{dim}(G(3,8))=15>\text{dim}G(t,7) ~\text{for all}~ t\ge 1.$$
Since the Picard number of $G(3,n)~(n=6,7,8)$ is one, the only morphisms $$G(3,6)\rightarrow G(t,5),$$ $$G(3,7)\rightarrow G(t,6)$$ and $$G(3,8)\rightarrow G(t,7)$$ are all constant for any integer $t\ge 1$.\\

 Case \uppercase\expandafter{\romannumeral 3}. $\mathcal{M}_x$ is spinor variety $\mathcal{S}_n~(n=3,4,5,6)$. The Chow ring of $\mathcal{S}_n$ is presented as a quotient of $\mathbb{Z}[X_1,\ldots,X_n]$ module the relations
 $$X_s^2+2\sum_{i=1}^{s-1}(-1)^i X_{s+i}X_{s-i}+(-1)^s X_{2s}=0$$
for $1\le s\le n$, where $X_j^{,}s$ are the Schubert classes of codimension $j$, $X_0=1$ and $X_j=0$ for $j<0$ or $j>n$ (see Section 3.2 in \cite{Ha}). In particular, $\text{dim} H^{2t}(\mathcal{M}_x,\mathbb{C})=1$ for every $t\le [\frac{5}{2}]$. Hence, the only morphisms $\mathcal{M}_x\rightarrow G(t,5)$ are constant for any integer $1\le t\le [\frac{5}{2}]$. Remarkably, for $n=4,5,6$, the value of $\varsigma$ can be appropriately enlarged. Due to the dimension of $\mathcal{S}_n$ is $\frac{n(n+1)}{2}$, one can check that
$$\text{dim}(\mathcal{S}_4)=10>\text{dim}G(t,6),$$
$$\text{dim}(\mathcal{S}_5)=15>\text{dim}G(t,7) $$ and
$$\text{dim}(\mathcal{S}_6)=21>\text{dim}G(t,9) ~\text{for all}~ t\ge 1.$$
Since the Picard number of $\mathcal{S}_n$ is one, the only morphisms $$\mathcal{S}_4\rightarrow G(t,6),$$ $$\mathcal{S}_5\rightarrow G(t,7)$$ and $$\mathcal{S}_6\rightarrow G(t,9)$$ are all constant for any integer $t\ge 1$.
\\
\\
Case \uppercase\expandafter{\romannumeral 4}. $\mathcal{M}_x=E_6/P_6$. The Chow ring of $E_6/P_6$ have the following form (see \cite{HD} Theorem 5). Let $y_1,y_4$ be the Schubert classes on $E_6/P_6$. Then
$$A(E_6/P_6)=\mathbb{Z}[y_1,y_4]/(r_9,r_{12}),$$
where $$r_9=2y_1^9+3y_1y_4^2-6y_1^5y_4;$$
$$r_{12}=y_4^3-6y_1^4y_4^2+y_1^{12}.$$
\begin{lemma}
	There are no nonconstant maps from $E_6/P_6$ to $G(t,10)$ for any integer $1\le t\le 5$.
\end{lemma}
\begin{proof}
(i). $1\le t\le 3$. Since $\text{dim} H^{2t}(\mathcal{M}_x,\mathbb{C})=1$, by the proof of Lemma 3.4 in paper \cite{M-O-C2}, the only morphisms $E_6/P_6\rightarrow G(t,10)$ are constant.

(ii). $t=4$. Let $\phi$ be a morphism from $E_6/P_6$ to $G(4,10)$. On $E_6/P_6$, we have an exact sequence\[
0\rightarrow\phi^{\ast}H_4\rightarrow\mathcal{O}^{\oplus 10}_{G(d,n)}\rightarrow\phi^{\ast}Q_6\rightarrow 0.
\]
Then $c(\phi^{\ast}H_4)\cdot c(\phi^{\ast}Q_6)=1$. According to the Chow ring of $E_6/P_6$, we can expand the equation into the following form:
\begin{align*}
(1+a_1y_1+a_2y_1^2+a_3y_1^3+&a_4y_1^4+\widetilde{a_4}y_4)\cdot(1+b_1y_1+b_2y_1^2+\\
&b_3y_1^3+b_4y_1^4+\widetilde{b_4}y_4+b_5y_1^5+\widetilde{b_5}y_1y_4+b_6y_1^6+\widetilde{b_6}y_1^2y_4)=1.
\end{align*}
Since $A^4(E_6/P_6)$ is freely generated by the classes $y_1^4,y_4$, the above equation implies that the coefficient of $y_4$ is $0$, i.e. $\widetilde{a_4}+\widetilde{b_4}=0$. On the other hand, $A^8(E_6/P_6)$ is freely generated by the classes $y_1^8,y_1^4y_4,y_4^2$, the above equation implies that the coefficient of $y_4^2$ is also zero, i.e. $\widetilde{a_4}\cdot\widetilde{b_4}=0$. Hence $\widetilde{a_4}=\widetilde{b_4}=0$. Therefore, this case can boil down to case (i).

(iii). $t=5$. By iterating the previous process, we get $c(\phi^{\ast}H_5)\cdot c(\phi^{\ast}Q_5)=1$, i.e.
\begin{align*}
(1+a_1y_1+a_2y_1^2+a_3y_1^3&+a_4y_1^4+\widetilde{a_4}y_4+a_5y_1^5+\widetilde{a_5}y_1y_4)\\
\cdot&(1+b_1y_1+b_2y_1^2+b_3y_1^3+b_4y_1^4+\widetilde{b_4}y_4+b_5y_1^5+\widetilde{b_5}y_1y_4)=1.
\end{align*}
In a similar way, we can prove $\widetilde{a_4}=\widetilde{b_4}=0$. Next, let's consider the vanishing of $\widetilde{a_5}$. Since $A^5(E_6/P_6)$ is freely generated by the classes $y_1^5,y_1y_4$, the above equation implies that the coefficient of $y_1y_4$ is $0$, i.e. $\widetilde{a_5}+\widetilde{b_5}=0$. On the other hand, $A^6(E_6/P_6)$ is freely generated by the classes $y_1^6,y_1^2y_4$, the above equation implies that the coefficient of $y_1^2y_4$ is also zero, i.e. $a_1\widetilde{b_5}+b_1\widetilde{a_5}=0$. Combining these equations with $a_1=-b_1$, we obtain that $\widetilde{a_5}=\widetilde{b_5}=0$. Therefore, this case can boil down to case (i).
\end{proof}

Case \uppercase\expandafter{\romannumeral 5}. $\mathcal{M}_x=E_7/P_7$. The Chow ring of $E_7/P_7$ have the following form (see \cite{HD} Theorem 6). Let $y_1,y_5,y_9$ be the Schubert classes on $E_7/P_7$. Then
$$A(E_7/P_7)=\mathbb{Z}[y_1,y_5,y_9]/(r_{10},r_{14},r_{18}),$$
where $$r_{10}=y_5^2-2y_1y_9;$$
$$r_{14}=2y_5y_9-9y_1^4y_5^2+6y_1^9y_5-y_1^{14};$$
$$r_18=y_9^2+10y_1^3y_5^3-9y_1^8y_5^2+2y_1^{13}y_5.$$

\begin{lemma}
	There are no nonconstant maps from $E_7/P_7$ to $G(t,13)$ for any integer $1\le t\le 6$.
\end{lemma}
\begin{proof}
	(i). $1\le t\le 4$. Since $\text{dim} H^{2t}(\mathcal{M}_x,\mathbb{C})=1$, by the proof of Lemma 3.4 in paper \cite{M-O-C2}, the only morphisms $E_7/P_7\rightarrow G(t,13)$ are constant.
	
	(ii). $t=5$. Let $\phi$ be a morphism from $E_7/P_7$ to $G(5,13)$. On $E_7/P_7$, we have an exact sequence\[
	0\rightarrow\phi^{\ast}H_5\rightarrow\mathcal{O}^{\oplus 13}_{G(d,n)}\rightarrow\phi^{\ast}Q_8\rightarrow 0.
	\]
	Then $c(\phi^{\ast}H_5)\cdot c(\phi^{\ast}Q_8)=1$. According to the Chow ring of $E_7/P_7$, we can expand the equation into the following form:
	\begin{align*}
	(1+a_1y_1+&a_2y_1^2+a_3y_1^3+a_4y_1^4+a_5y_1^5+\widetilde{a_5}y_5)\cdot(1+b_1y_1+b_2y_1^2+b_3y_1^3\\
&+b_4y_1^4+b_5y_1^5+\widetilde{b_5}y_5+b_6y_1^6+\widetilde{b_6}y_1y_5+b_7y_1^7+\widetilde{b_7}y_1^2y_5+b_8y_1^8+\widetilde{b_8}y_1^3y_5)=1.
	\end{align*}
	Since $A^5(E_7/P_7)$ is freely generated by the classes $y_1^5,y_5$, the above equation implies that the coefficient of $y_5$ is $0$, i.e. $\widetilde{a_5}+\widetilde{b_5}=0$. On the other hand, $A^{10}(E_7/P_7)$ is freely generated by the classes $y_1^{10},y_1^5y_5,y_5^2$, the above equation implies that the coefficient of $y_5^2$ is also zero, i.e. $\widetilde{a_5}\widetilde{b_5}=0$. Hence $\widetilde{a_5}=\widetilde{b_5}=0$. Therefore, this case can boil down to case (i).
	
	(iii). $t=6$. By iterating the previous process, we get $c(\phi^{\ast}H_6)\cdot c(\phi^{\ast}Q_7)=1$, i.e.
	\begin{align*}	(1+a_1y_1+a_2y_1^2&+a_3y_1^3+a_4y_1^4+a_5y_1^5+\widetilde{a_5}y_5+a_6y_1^6+\widetilde{a_6}y_1y_5)\cdot(1+b_1y_1+\\
&b_2y_1^2+b_3y_1^3+b_4y_1^4+b_5y_1^5+\widetilde{b_5}y_5+b_6y_1^6+\widetilde{b_6}y_1y_5+b_7y_1^7+\widetilde{b_7}y_1^2y_5)=1.
	\end{align*}
	In a similar way, we can prove $\widetilde{a_5}=\widetilde{b_5}=0$. Next, let's consider the vanishing of $\widetilde{a_6}$. Since $A^6(E_7/P_7)$ is freely generated by the classes $y_1^6,y_1y_5$, the above equation implies that the coefficient of $y_1y_5$ is $0$, i.e. $\widetilde{a_6}+\widetilde{b_6}=0$. On the other hand, $A^{12}(E_7/P_7)$ is freely generated by the classes $y_1^{12},y_1^2y_5^2,y_1^7y_5$, the above equation implies that the coefficient of $y_1^2y_5^2$ is also zero, i.e. $\widetilde{a_6}\widetilde{b_6}=0$. Hence, $\widetilde{a_6}=\widetilde{b_6}=0$. Therefore, this case can boil down to case (i).
\end{proof}

Case \uppercase\expandafter{\romannumeral 6}. $\mathcal{M}_x=C_3/P_3$. The Chow ring of $C_3/P_3$ is $$A(C_3/P_3)=\mathbb{Z}[y_1,y_3]/(y_1^4-8y_1y_3,x_3^2),$$
where $y_1,y_3$ are the Schubert classes on $C_3/P_3$ (see Section 3.1 in \cite{Ha}). Since $\text{dim} H^{2t}(C_3/P_3,\mathbb{C})=1$ for $1\le t\le 2$, the only morphisms $C_3/P_3\rightarrow G(t,5)$ are constant for any integer $1\le t\le 2$.

Summing up, we have obtained the following Table \ref{bound}.

\begin{table}[htbp]
	\caption{$\varsigma(\mathcal{G})$} \label{bound}
	\begin{center}
	\begin{tabular}{|c|c|c|c|c|c|c|c|c|c|}
		\hline
		$\mathcal{G}$ &\tabincell{c}{$A_n/P_1$\\$A_n/P_n$}  & \tabincell{c}{$B_n/P_1$\\$B_n/P_n$} &  \tabincell{c}{$C_n/P_1$\\$C_n/P_n$} &
  \tabincell{c}{$D_n/P_1$\\$D_n/P_{n-1}$\\$D_n/P_n$} &
 \tabincell{c}{$E_6/P_1$\\$E_7/P_1$\\$E_8/P_1$} & \tabincell{c}{$E_6/P_2$\\$E_7/P_2$\\$E_8/P_2$}  &\tabincell{c}{$E_6/P_6$\\$E_7/P_7$\\$E_8/P_8$}  &\tabincell{c}{$F_4/P_1$\\$F_4/P_4$} & \tabincell{c}{$G_2/P_1$\\$G_2/P_2$}\\
		\hline
		\hline
		$\varsigma(\mathcal{G})$ & \tabincell{c}{$n-1$\\$n-1$} &\tabincell{c}{$2n-3$\\ $n-1$}& \tabincell{c}{$2n-2$\\$n-1$} & \tabincell{c}{$2n-5$\\$n-1$\\$n-1$ } &\tabincell{c}{6\\7\\9}&\tabincell{c}{5\\6\\7} &\tabincell{c}{6\\10\\13}& \tabincell{c}{5\\5} & \tabincell{c}{3\\1}\\
		\hline
	\end{tabular}
\end{center}
\end{table}

\begin{thm}\label{GGS}
Let $E$ be a uniform $r$-bundle on a generalized Grassmannian $\mathcal{G}$ with extremal node marked. If $r\le \varsigma(\mathcal{G})$, then $E$ splits as a direct sum of line bundles.
\end{thm}

\begin{rem}\label{rem}
One can find the value of $\varsigma(\mathcal{G})$ is equal to or bigger than the value of $\varsigma$ that we analyzed such that the morphism $\mathcal{M}_x\rightarrow G(t,\varsigma)$ can only be constant. $\varsigma(\mathcal{G})$ is equal to $\varsigma$ except $\mathcal{G}=C_n/P_1$ and $G_2/P_1$. The reason is that uniform $2n-2$ bundles split on $C_n/P_1\cong\mathbb{P}_{2n-1}$ and uniform $3$ bundles split on $G_2/P_1\cong Q^5$ by the previous arguments. In all cases, however, the morphism $\mathcal{M}_x\rightarrow G(t,\varsigma(\mathcal{G})-2)$ can only be constant.
\end{rem}

\begin{cor}
If $\mathcal{G}$ is a generalized Grassmannian with marked point $k$, $k$ is not extremal, then $\mathcal{M}_x$ is a product of rational homogeneous spaces. If $$\mathcal{M}_x=A_1\times\cdots\times A_t~(2\le t\le 3)$$ and $$r\le \varsigma(\mathcal{G}):=\text{min}\{\varsigma(A_1'), \ldots,\varsigma(A_t')\},$$ where $A_i~(1\le i\le t)$ is the special family of lines of class $\check{\alpha_k}$ through $x$ on the generalized Grassmannian $A_i'$ with extremal node marked. Then uniform $r$-bundle $E$ on $\mathcal{G}$ splits as a direct sum of line bundles.
\end{cor}
\begin{proof}
In order to prove $E$ splits, we just need to prove that the only morphisms $\mathcal{M}_x\rightarrow G(t,r)$ are constant. Let $\phi$ be a morphism from $\mathcal{M}_x$ to $G(t,r)$. Because every $A_i~(1\le i\le t)$ can be regarded as the subspace of $\mathcal{M}_x$ corresponding to the special family of lines of class $\check{\alpha_k}$ through $x$ on $A_i'$, we can consider the restriction of $\phi$ to $A_i~(1\le i\le t)$. By assumption $r\le \text{min}\{\varsigma(A_1'), \ldots,\varsigma(A_t')\}$, all the restriction maps are constant. Hence, $\phi$ is also constant.
\end{proof}
\begin{rem}
\emph{\begin{itemize}
\item For the case where $\mathcal{M}_x=\mathbb{P}^1$ or $\mathcal{M}_x=\mathbb{P}^1\times\cdots$, we can say nothing about the splitting result according to our theorem.
\item In some cases, the value of $\varsigma(\mathcal{G})$ cannot be expanded anymore, which means that there exist uniform but nonsplitting $\varsigma(\mathcal{G})+1$-bundles. For instance, Grassmannian $A_{n-1}/P_d=G(d,n)~(d\le n-d)$ has uniform but nonsplitting $d$-bundle $H_d$ (the universal bundle of $G(d,n)$); Spinor variety $D_{n+1}/P_{n+1}=\mathcal{S}_n$ has uniform but nonsplitting $n+1$-bundle $Q_{n+1}$ (the universal quotient bundle of $\mathcal{S}_n$).
\item Compare to the main theorem (Theorem 3.1) in \cite{M-O-C2}, we improve their results. In particular, we enlarge the splitting threshold for uniform bundles on Hermitian symmetric spaces $E_6/P_6$ from $\ge 5$ to $\ge 6$ and $E_7/P_7$ from $\ge 7$ to $\ge 10$ (see Table 1 in \cite{M-O-C2}).
\end{itemize}}
\end{rem}

\begin{cor}\label{d}
Let $\mathcal{G}$ be a generalized Grassmannian covered by linear projective subspaces of dimension $2$ and $E$ be an $r$-bundle on $\mathcal{G}$. If $E$ splits as a direct sum of line bundles when it restricts to every $\mathbb{P}^2\subseteq \mathcal{G}$,  then $E$ splits as a direct sum of line bundles on $\mathcal{G}$.
\end{cor}
\begin{proof}
If $\mathcal{G}$ is a projective space, then the result holds (\cite{O-S-S} Theorem 2.3.2). Suppose $\mathcal{G}=G/P_k$ which is not a projective space and $k$ is the unique black node, where $G$ is a simple Lie group and $P_k$ is a maximal parabolic subgroup of $G$. The condition implies that $E$ is uniform.
The reason for this is the fact every line $L$ is contained in two different $\mathbb{P}^2$ by Theorem \ref{hdp}.

We prove the corollary by induction on $r$. If we have the exact sequence of vector bundles
\begin{align}\label{a}
0\rightarrow M\rightarrow E\rightarrow N\rightarrow 0
\end{align}	
on $\mathcal{G}$,
where the rank of $M$ and $N$ are smaller than $r$,
such that \[
M|_Z=\bigoplus\limits_{i=1}^{r-t} \mathcal{O}_Z(a_{t+i}),N|_Z=\mathcal{O}_Z^{\oplus t},
\]
 for every $Z\simeq \mathbb{P}^2$, then by the induction hypothesis, $M$ and $N$ split. Since $H^1(\mathcal{G},N^{\vee}\otimes M)=0$, the above exact sequence splits and hence also $E$.

Similar to the proof of Theorem 3.3 in \cite{D-F-G}, on $\mathcal{U}=G/(P_k\cap P_{N(k)})$, we can obtain an exact sequence
\[
0\rightarrow\widetilde{M}\rightarrow {q_1}^{\ast}E\rightarrow\widetilde{N}\rightarrow0.	\]
If we prove that the morphism $\varphi$ is constant for every $x\in \mathcal{G}$, then there exist two bundles $M$, $N$ over $G$ with $\widetilde{M}={q_1}^{\ast}M,\widetilde{N}={q_1}^{\ast}N$. By projecting the bundle sequence\[
0\rightarrow {q_1}^{\ast}M\rightarrow {q_1}^{\ast}E\rightarrow {q_1}^{\ast}N\rightarrow0
\]
onto $\mathcal{G}$, we can get the desired exact sequence (\ref{a}).
Thus, to prove the existence of the above exact sequence, it suffices to show that the map\[
 \varphi:\mathcal{M}_x\rightarrow\mathbb{G}_k(t-1,\mathbb{P}_k(E^{\vee}_x))
 \]
 is constant for every $x\in\mathcal{G}$ .
 Given a projective subspace $Z$ of dimension $2$ and a line $L\subseteq Z$, we take any point $x\in L$ and denote by $Z'$ the subspace of $\mathcal{M}_x$ corresponding to the tangent directions to $Z$ at $x$. By the hypothesis, $E|Z$ is a direct sum of line bundles, so
 \[
  \varphi|_{Z'}:Z'\rightarrow\mathbb{G}_k(t-1,\mathbb{P}_k(E^{\vee}_x))
 \]
 is constant. Since $\mathcal{G}$ covered by linear projective subspaces of dimension $2$ and $M_x$ is chain-connected by $\mathbb{P}^1$, $\varphi$ is constant for every $x\in \mathcal{G}$.
\end{proof}


\subsection{Uniform vector bundles on rational homogeneous spaces}
Let \[X=G/P\simeq G_1/P_{I_1}\times G_2/P_{I_2}\times \cdots \times G_m/P_{I_m}, \] where $G_i$ is a simple Lie group with \emph{Dynkin diagram} $\mathcal{D}_i$ whose set of nodes is $D_i$ and $P_{I_i}$ is a parabolic subgroup of $G_i$ corresponding to $I_i\subset D_i$. We set $F(I_i):=G_i/P_{I_i}$ by marking on the Dynkin diagram $\mathcal{D}_i$ of $G_i$ the nodes corresponding to $I_i$. Let $\delta_i$ be a node in $\mathcal{D}_i$ and $N(\delta_i)$ be the set of nodes in $\mathcal{D}_i$ that are connected to $\delta_i$.

If $\delta_i\in I_i$, we call $$\mathcal{M}_i^{\delta_i^c}:=G_i/P_i^{\delta_i^c}\times \widehat{G_i/P_{I_i}}~ (1\le i\le m),$$ the \emph{$i$-th special family} of lines of class $\check{\delta}_i$ by Theorem \ref{lines}, where $P_i^{\delta_i^c}:=P_{I_i\backslash \delta_i\cup N(\delta_i)}$ and $\widehat{G_i/P_{I_i}}$ is $G_1/P_{I_1}\times G_2/P_{I_2}\times \cdots \times G_m/P_{I_m}$ by deleting $i$-th term $G_i/P_{I_i}$. For $i=1$ and $\delta\in I$, we will use the notation $\mathcal{M}^{\delta}$ to denote the special family of lines of class $\check{\delta}$.

For $x\in X$, we call
\[\text{SPVMRT}_x^{(\delta_i)}=\{L\in \mathcal{M}_i^{\delta_i^c}|x\in L\}\] the \emph{$\delta_i$-th special part of variety of minimal rational tangents at $x$} (sometimes we just write $\text{SPVMRT}_x$ if there is no confusion).

Fix $\delta_i\in I_i$. $\text{SPVMRT}_x^{(\delta_i)}$ is just the special family of lines of class $\check{\delta}_i$ through $x$ on the generalized Grassmannian $\mathcal{G}^{\delta_i}$ whose Dynkin diagram $\mathcal{D}^{\delta_i}$ is the maximal sub-diagram of $(\mathcal{D}_i,I_i)$ with the only marked point $\delta_i$. Denote $\nu(X, \delta_i):=\varsigma(\mathcal{G}^{\delta_i})$.
Let\[\nu(X):=\text{min}_{i}\{\text{min}_{\delta_i\in I_i}\{\nu(X, \delta_i)\}\}.\]

\begin{definition}
 A vector bundle $E$ on $X$ is called poly-uniform with respect to $\mathcal{M}_i^{\delta_i^c}$ for every $i~(1\le i\le m)$ and every $\delta_i\in I_i$ if the restriction of $E$ to every line in $\mathcal{M}_i^{\delta_i^c}$ has the same splitting type.
 We also call that $E$ poly-uniform with respect to all the special families of lines.
\end{definition}

Let $\mathscr{F}$ be a torsion free coherent sheaf of rank $r$ over $X$. Fix integer $i~(1\le i \le m)$ and $\delta_i\in I_i$.
Since the singularity set $S(\mathscr{F})$ of $\mathscr{F}$ has codimension at least $2$, there are lines $L\in \mathcal{M}_i^{\delta_i^c}$ which do not meet $S(\mathscr{F})$. If
\[\mathscr{F}|L\cong \mathcal{O}_{L}(a_1^{(\delta_i)})\oplus\cdots\oplus \mathcal{O}_{L}(a_r^{(\delta_i)}).\]
Let \[c_1^{(\delta_i)}(\mathscr{F})=a_1^{(\delta_i)}+\cdots+a_r^{(\delta_i)},\] which is independent of the choice of $L$.
We set
\[\mu^{(\delta_i)}(\mathscr{F})=\frac{c_1^{(\delta_i)}(\mathscr{F})}{\text{rk}(\mathscr{F})}.\]

\begin{definition}\label{i}
A torsion free coherent sheaf $\mathscr{E}$ over $X$ is $\delta_i$-semistable ($\delta_i$-stable) if for every coherent subsheaf $ \mathscr{F}\subseteq \mathcal{E}$ with $0< \text{rk}(\mathscr{F})< \text{rk}(\mathcal{E})$, we have
\[\mu^{(\delta_i)}(\mathscr{F})\le~ (<)~ \mu^{(\delta_i)}(\mathcal{E}).\]
If $E$ is not $\delta_i$-semistable, then we call $E$ is $\delta_i$-unstable.
\end{definition}

\begin{prop}\label{notsame}
Fix integer $i~(1\le i \le m)$ and $\delta_i\in I_i$. Let $E$ be a uniform $r$-bundle on $X$ of type $(a_1^{(\delta_i)}, \ldots ,a_r^{(\delta_i)}),~a_1^{(\delta_i)}\le \ldots \le a_r^{(\delta_i)}$ with respect to $\mathcal{M}_i^{\delta_i^c}$. If $r\le \nu(X,\delta_i)-2$ and these $a_j^{(\delta_i)}$'s are not all same, then $E$ can be expressed as an extension of uniform bundles with respect to $\mathcal{M}_i^{\delta_i^c}$. In particular, $E$ is $\delta_i$-unstable.
	\end{prop}
	\begin{proof}
After twisting with an appropriate line bundle, we can assume that $E$ has the splitting type
\[\underline{a}_E^{(\delta_i)}=(0,\ldots,0,a_{t+1}^{(\delta_i)},\ldots,a_r^{(\delta_i)}),~a_{t+i}^{(\delta_i)}>0, ~\text{for}~i=1,\ldots,r-t.
\]
with respect to $\mathcal{M}_i^{\delta_i^c}$.

Let's consider the standard diagram
\begin{align}\label{key}
\xymatrix{
	\mathcal{U}_i^{\delta_i^c}=G_i/P_{I_i\cup N(\delta_i)}\times \widehat{G_i/P_{I_i}}\ar[d]^{q_1}   \ar[r]^-{q_2} & \mathcal{M}_i^{\delta_i^c} \\
	X.
}
\end{align}
For $L\in \mathcal{M}_i^{\delta_i}$, the $q_2$-fiber
\[\widetilde{L}={q_2}^{-1}(L)=\{(x,L)\in X\times \mathcal{M}_i^{\delta_i^c}|x\in L\},\]
is mapped under $q_1$ to the line $L$ identically in $G/P$ and we have \[{q_1}^\ast E|_{\widetilde{L}}\cong E|_L.\]
For $x\in G/P$, the $q_1$-fiber ${q_1}^{-1}(x)$
is mapped isomorphically under $q_2$ to the subvariety
\[\text{SPVMRT}_x=\{L\in \mathcal{M}_i^{\delta_i^c}|x\in L\}.\]
Because \[E|_L\cong\mathcal{O}_L^{\oplus t}\oplus\bigoplus\limits_{i=1}^{r-t} \mathcal{O}_L(a_{t+i}^{(\delta_i)}),~ a_{t+i}^{(\delta_i)}>0,\]
\[h^0\left({q_2}^{-1}(L),{q_1}^\ast (E^{\vee})|_{q_2^{-1}(L)}\right)=t\] for all $L\in \mathcal{M}_i^{\delta_i^c}$.
Thus the direct image ${q_2}_{\ast}{q_1}^{\ast}(E^{\vee})$ is a vector bundle of rank $t$ over $\mathcal{M}_i^{\delta_i^c}$.
The canonical homomorphism of sheaves\[
{q_2}^{\ast}{q_2}_{\ast}{q_1}^{\ast}(E^{\vee})\rightarrow {q_1}^{\ast}(E^{\vee})
\]
makes $\widetilde{N}^{\vee}:={q_2}^{\ast}{q_2}_{\ast}{q_1}^{\ast}(E^{\vee})$ to be a subbundle of ${q_1}^{\ast}(E^{\vee})$. Because over each ${q_2}$-fiber $\widetilde{L}$, the evaluation map \[
\widetilde{N}^{\vee}|_{\widetilde{L}}=H^0(\widetilde{L},{q_1}^{\ast}(E^{\vee})|_{\widetilde{L}})\otimes_{k}\mathcal{O}_{\widetilde{L}}\rightarrow {q_1}^{\ast}(E^{\vee})|_{\widetilde{L}}
\]
identifies $\widetilde{N}^{\vee}|_{\widetilde{L}}$ with $\mathcal{O}_L^{\oplus t}\subseteq\mathcal{O}_L^{\oplus t}\oplus\bigoplus_{i=1}^{r-t} \mathcal{O}_L(-a_{t+i}^{(\delta_i)})=E^{\vee}|_L.$
Over $\mathcal{U}_i^{\delta_i^c}$ we thus obtain an exact sequence\[
0\rightarrow\widetilde{M}\rightarrow {q_1}^{\ast}E\rightarrow\widetilde{N}\rightarrow0
\]
of vector bundles, whose restriction to ${q_2}$-fibers $\widetilde{L}$ looks as follows:
$$
\xymatrix{
0\ar[r] &\widetilde{M}|_{\widetilde{L}} \ar[r] \ar[d]_\cong&	{q_1}^{\ast}E|_{\widetilde{L}} \ar[r] \ar[d]_\cong&\widetilde{N}|_{\widetilde{L}}\ar[r]\ar[d]_\cong&0\\
	0\ar[r] &\bigoplus_{i=1}^{r-t} \mathcal{O}_L(a_{t+i}^{(\delta_i)})\ar[r]  &\mathcal{O}_L^{\oplus t}\oplus\bigoplus_{i=1}^{r-t} \mathcal{O}_L(a_{t+i}^{(\delta_i)}) \ar[r] &\mathcal{O}_L^{\oplus t}\ar[r]&0.\\
	 }
$$
Because $\widetilde{N}^{\vee}$ is a subbundle of ${q_1}^{\ast}(E^{\vee})$ of rank $t$, for every point $x\in G$, it provides a morphism\[
\varphi:\text{SPVMRT}_x\rightarrow\mathbb{G}_k(t-1,\mathbb{P}_k(E^{\vee}_x)).
\]
Since $r\le \nu(X,\delta_i)-2$, morphism $\varphi$ is constant by Remark \ref{rem} and the definition of $\nu(X, \delta_i)$. It follows that $\widetilde{M}$ and $\widetilde{N}$ are trivial on all ${q_1}$-fibers. So the canonical morphisms ${q_1}^{\ast}{q_1}_{\ast}\widetilde{M}\rightarrow\widetilde{M}$ and ${q_1}^{\ast}{q_1}_{\ast}\widetilde{N}\rightarrow\widetilde{N}$ are isomorphisms. Hence there are uniform bundles $M={q_1}_{\ast}\widetilde{M}$, $N={q_1}_{\ast}\widetilde{N}$ with respect to $\mathcal{M}_i^{\delta_i^c}$ over $X$ with \[\widetilde{M}={q_1}^{\ast}M,\widetilde{N}={q_1}^{\ast}N.
\]
By projection formula and the rationality of the $q_1$ fiber SPVMRT$_x$, we project the bundle sequence\[
0\rightarrow {q_1}^{\ast}M\rightarrow {q_1}^{\ast}E\rightarrow {q_1}^{\ast}N\rightarrow0
\]
onto $X$ to get the exact sequence
\begin{align}\label{c}
0\rightarrow M\rightarrow E\rightarrow N\rightarrow 0.
\end{align}
So $\mu^{(\delta_i)}(N)>\mu^{(\delta_i)}(E)$ and thus $E$ is $\delta_i$-unstable
	\end{proof}

\begin{prop}\label{polyuni}
On $X$, if an $r$-bundle $E$ is poly-uniform with respect to all the special families of lines on $X$ such that the splitting type with respect to $\mathcal{M}_i^{\delta_i^c}$ is $(a^{(\delta_i)}, \ldots, a^{(\delta_i)})$ for each $i$ and $\delta_i$, then $E$ splits as a direct sum of line bundles.
\end{prop}
\begin{proof}
After twisting with an appropriate line bundle, we can assume that $E$ is trivial on all the special families of lines on $X$. We are going to show that $E$ is trivial.

Let's first consider the case that $X$ is a generalized flag manifold, which corresponds to a connected marked Dynkin diagram $\mathcal{D}(X)$ with $l$ black nodes. We prove the lemma by induction on $l$. For $l=1$, $X$ is just a generalized Grassmannian. Then the result holds by \cite{A-W} Proposition 1.2. Suppose the assertion is true for all generalized flag manifolds with connected marked Dynkin diagram and $l'$ black nodes ($1\le l'<l$). Let's consider the natural projection
\[\pi:X\rightarrow X',\]
where $X'$ is corresponding to the marked Dynkin diagram $\mathcal{D}(X)$ by changing the first black node $\delta$ to white. It's not hard to see that every $\pi$-fiber $\pi^{-1}(x)$ is isomorphic to the generalized Grassmannian $\mathcal{G}^{\delta}$ with the only marked point $\delta$. Since the restriction of $E$ to every line in $\mathcal{M}^{\delta}$ is trivial, so is to every line in $\pi^{-1}(x)$. Thus $E$ is trivial on all $\pi$-fibers by \cite{A-W} Proposition 1.2. It follows that $E'=\pi_{*}E$ is an algebraic vector bundle of rank $r$ over $X'$ and $E\cong \pi^{*}E'$.
	
\textbf{Claim.} $E'$ is trivial on all the special families of lines on $X'$.

In fact, let $\gamma$ be a black node in $\mathcal{D}(X)$ that is different from $\delta$ and $L$ be a line in $\mathcal{M}^{\gamma}$. Then $\pi(L)$ is a line in the family of lines of $X'$. When $\gamma$ runs through all black nodes except $\delta$ and $L$ runs through all lines in $\mathcal{M}^{\gamma}$ in $X$, $\pi(L)$ also runs through all lines in all the special families of lines of $X'$. The projection $\pi$ induces an isomorphism\[
E'|_{\pi(L)}\cong \pi^{*}E'|_L\cong E|_L.
\]
We identify $L$ with $\pi(L)$. Since $E|_L$ is trivial for every line $L$ in $\mathcal{M}^{\gamma}$ by assumption, $E'|_L$ is trivial for every line $L$ in all the special families of lines of $X'$. By the induction hypothesis, $E'$ is trivial. Thus $E\cong q^{*}E'$ is trivial.

Now let's think about the general case, where the marked Dynkin diagram of $X$ is not connected.
Assume $X$ can be decomposed into a product
\[X=G/P\simeq G_1/P_{I_1}\times G_2/P_{I_2}\times \cdots \times G_m/P_{I_m}, \] where $m\ge 2$ and $G_i~(1\le i\le m)$ is a simple Lie group with connected Dynkin diagram. We prove $E$ is trivial by induction on $m$. For $m=1$, the result holds from the previous analysis. Consider the natural projection
\[f:X\rightarrow X':=G_2/P_{I_2}\times \cdots \times G_m/P_{I_m},\]
it's easy to see that every $f$-fiber $f^{-1}(x)$ is isomorphic to $G_1/P_{I_1}$. By assumption, $E$ is trivial on all the special families of lines on $f^{-1}(x)\cong G_1/P_{I_1}$. Thus $E$ is trivial on all $f$-fibers by the previous analysis. It follows that $E'=f_{*}E$ is an algebraic vector bundle of rank $r$ over $X'$ and $E\cong f^{*}E'$. Similarly, we can prove that $E'$ is trivial and thus $E$ is trivial.
\end{proof}

\begin{thm}
On $X$, if $r$-bundle $E$ is poly-uniform with respect to all the special families of lines and $r\le \nu(X)-2$, then $E$ is $\delta_i$-unstable for some $\delta_i$ ($1\le i \le m$) or $E$ splits as a direct sum of line bundles.
\end{thm}
\begin{proof}
It is obviously from Proposition \ref{notsame} and \ref{polyuni}.
\end{proof}


\section{Semistable vector bundles on rational homogeneous spaces}
 Let $G$ be a simple Lie group and $D=\{1,2,\ldots,n\}$ be the set of nodes of the \emph{Dynkin diagram $\mathcal{D}$ of $G$}. Denote by $P_k$ the parabolic subgroup of $G$ corresponding to the node $k$. Consider $X=G/P_k$ in its minimal homogeneous embedding. Denote by $\mathcal{M}:=G/P_{N(k)}$ the generalized flag manifold defined by the marked Dynkin diagram $(\mathcal{D},N(k))$ and by $\mathcal{U}:=G/P_{k,N(k)}=G/(P_{N(k)}\cap P_k)$ the universal family, which has a natural $\mathbb{P}^{1}$-bundle structure over $\mathcal{M}$, i.e. we have the natural diagram:
\begin{align}
	\xymatrix{
		\mathcal{U}\ar[d]^{p}   \ar[r]^-{q} & \mathcal{M}\\
		X=G/P_k.
	}
\end{align}	
Given $x\in X$, $\mathcal{M}_x=q(p^{-1}(x))$ coincides with $H/P_{N(k)}$ by Remark \ref{vmrt} where the set of nodes of the Dynkin diagram $H$ is $D(H)=D\backslash {k}$ .

In the above setting, we will show that the splitting type of $T_{\mathcal{U}/X}|_{q^{-1}(l)}$ take the form $(-1,\ldots,-1)$, $(-1,\ldots,-1,-2,\ldots,-2)$ or $(-3)$. Building upon this assert, we will generalize the Grauert-M$\ddot{\text{u}}$lich-Barth theorem to any rational homogeneous spaces.
\subsection{The classical simple Lie algebras}
We say $G$ is of classical type when its Dynkin diagram is of type $A_n$, $B_n$, $C_n$ or $D_n$. Because $X=G/P_k$ has clear geometric explanations, we can even write down the specific form of the relative tangent bundle $T_{\mathcal{U}/X}$. In the case of type $A_n$, the relative tangent bundle $T_{\mathcal{U}/X}$ is known (see \cite{D-F-G} Lemma 5.6) and the splitting type of $T_{\mathcal{U}/X}|_{q^{-1}(l)}$ is $(-1,\ldots,-1)$. So let's just consider the remaining three types.
\subsubsection{The marked Dynkin diagram $(B_n,k)$}
In this section, we consider the Dynkin diagram $B_n$, which corresponds to the classical Lie group $SO_{2n+1}$. Denote by $B_n/P_I:=SO_{2n+1}/P_I$ the generalized flag manifold with $I=\{i_1,\ldots,i_s\}\subseteq D$.  Let $V=\mathbb{C}^{2n+1}$ be a vector space equipped with a nondegenerate symmetric bilinear form $\mathcal{Q}$. Then $B_n/P_I$ is actually the \emph{odd Orthogonal flag manifold} $OG(i_1,\ldots,i_s;2n+1)$, which parametrizes flags
$$V_{i_1}\subset V_{i_2}\subset\cdots\subset V_{i_s}\subset V,$$
where each $V_{i_t}(1\le t\le s)$ is an $i_t$-dimensional isotropic subspace in $V$.

There is a universal flag of subbundles
$$0=H_0\subset H_{i_1}\subset \cdots\subset H_{i_s}\subset H_{i_s}^{\bot}\subset\cdots\subset H_{i_1}^{\bot}\subset \mathcal{O}_{B_n/P_I}\times V$$
on $B_n/P_I$, where $H_{i_t}^{\bot}$ is the \emph{$\mathcal{Q}$-orthogonal complement} of $H_{i_t}$, rank $H_{i_t}=i_t$ and rank $H_{i_t}^{\bot}=2n+1-i_t~(1\le t\le s)$.

(1) For $k=1$, the odd Orthogonal Grassmannian $B_n/P_1:=OG(1,2n+1)$ is just the quadric $Q^{2n-1}$. In this case, $\mathcal{M}=B_n/P_2$ and $\mathcal{U}=B_n/P_{1,2}$. We have the natural diagram
\begin{align}
\xymatrix{
	\mathcal{U}=B_n/P_{1,2}\ar[d]^{p}   \ar[r]^-{q} &\mathcal{M}=B_n/P_2\\
	X=B_n/P_1
}
\end{align}	
and $\mathcal{M}_x$ is $B_{n-1}/P_1$, i.e. the quadric $Q^{2n-3}$.
\begin{lemma}\label{b1}
	Let $\widetilde{L}=q^{-1}(l)\subset \mathcal{U}$ for $l\in \mathcal{M}$. For the relative tangent bundle $T_{\mathcal{U}/X}$, we have
	$$T_{\mathcal{U}/X}|_{\widetilde{L}}=\mathcal{O}_{\widetilde{L}}(-1)^{\oplus 2n-3}.$$
\end{lemma}
 \begin{proof}
 For $x\in X$, the $p$-fiber $p^{-1}(x)=\{(x,l)|x\in L\}$ is isomorphic to $\mathcal{M}_x=Q^{2n-3}$. Over $p^{-1}(x)\cong Q^{2n-3}$, we have the universal bundle sequence
\begin{align}\label{uni}
0\rightarrow \mathcal{O}_{Q^{2n-3}}(-1)\rightarrow \mathcal{O}_{Q^{2n-3}}(-1)^{\bot} \rightarrow \mathcal{O}_{Q^{2n-3}}(-1)^{\bot}/\mathcal{O}_{Q^{2n-3}}(-1) \rightarrow 0,
\end{align}
where $\mathcal{O}_{Q^{2n-3}}(-1)$ is the rank 1 tautological bundle over $Q^{2n-3}$, which is the pull back of $\mathcal{O}_{\mathbb{P}^{2n-2}}(-1)$ under the embedding $Q^{2n-3}\hookrightarrow \mathbb{P}^{2n-2}$ and $\mathcal{O}_{Q^{2n-3}}(-1)^{\bot}$ denotes its $\mathcal{Q}$-orthogonal complement. Notice that the tangent bundle of $Q^{2n-3}$ can be represented as
$$T_{Q^{2n-3}}=\mathcal{O}_{Q^{2n-3}}(1)\otimes\mathcal{O}_{Q^{2n-3}}(-1)^{\bot}.$$
Let's consider the exact sequence
$$0\rightarrow H_2/H_1\rightarrow H_2^{\bot}/H_1\rightarrow q^{*}(H_2^{\bot}/H_2)\rightarrow 0,$$ of vector bundles on $\mathcal{U}=B_n/P_{1,2}$.
By restricting it to the the $p$-fiber $p^{-1}(x)$, we obtain the universal bundle sequence (\ref{uni}) on $p^{-1}(x)$. Therefore,
$T_{\mathcal{U}/X}=(H_2/H_1)^{\vee}\otimes q^{*}(H_2^{\bot}/H_2)$
and
$$T_{\mathcal{U}/X}|_{\widetilde{L}}=\mathcal{O}_{\widetilde{L}}(-1)^{\oplus 2n-3}.$$
 \end{proof}

(2) For $k~(2\le k\le n-1)$, the odd Orthogonal Grassmannian $B_n/P_k:=OG(k,2n+1)$ parametrizes the $k$-dimensional isotropic subspaces in $V$. In this case, $\mathcal{M}=B_n/P_{k-1,k+1}$ and $\mathcal{U}=B_n/P_{k-1,k,k+1}$. We have the natural diagram:
\begin{align}
\xymatrix{
	\mathcal{U}=B_n/P_{k-1,k,k+1}\ar[d]^{p}   \ar[r]^-{q} &\mathcal{M}=B_n/P_{k-1,k+1}\\
	X=B_n/P_k
}
\end{align}	
and $\mathcal{M}_x$ is $\mathbb{P}^{k-1}\times B_{n-k}/P_1=\mathbb{P}^{k-1}\times Q^{2(n-k)-1}$.
\begin{lemma}\label{bk}
	Let $\widetilde{L}=q^{-1}(l)\subset \mathcal{U}$ for $l\in \mathcal{M}$. For the relative tangent bundle $T_{\mathcal{U}/X}$, we have
	$$T_{\mathcal{U}/X}|_{\widetilde{L}}=\mathcal{O}_{\widetilde{L}}(-1)^{\oplus 2n-k-2}.$$
\end{lemma}
\begin{proof}
It's not hard to check that over $\mathcal{U}$, we have the following two exact sequences of vector bundles:
\begin{align*}
0\rightarrow (H_k/H_{k-1})^{\vee}\rightarrow p^{\ast}H_k^{\vee} \rightarrow q^{\ast}H_{k-1}^{\vee} \rightarrow 0,
\end{align*}
\begin{align*}
0\rightarrow H_{k+1}/H_k\rightarrow H_{k+1}^{\bot}/H_k\rightarrow q^{*}(H_{k+1}^{\bot}/H_{k+1})\rightarrow 0.
\end{align*}
We will consider the following diagram
\begin{align}
\xymatrix{
\mathcal{U}_1=B_n/P_{k-1,k}  \ar[rrdd]_{p_1} &&\mathcal{U}=B_n/P_{k-1,k,k+1}\ar[ll]_{\pi_1}\ar[rr]^{\pi_2}\ar[dd]^{p}  && \mathcal{U}_2=B_n/P_{k,k+1}  \ar[lldd]^{p_2} \\
	\\
	&&X.
}
\end{align}
All the morphisms in the above diagram are projections. For any $x\in X$, the $p_1$-fiber $p_1^{-1}(x)$ is isomorphic to $\mathbb{P}^{k-1}$ and the $p_2$-fiber $p_2^{-1}(x)$ is isomorphic to $Q^{2(n-k)-1}$. Note that for $x\in X$, the projection $q$ induces an isomorphism
$$q|p^{-1}(x):p^{-1}(x)\rightarrow \mathcal{M}_x=\mathbb{P}^{k-1}\times Q^{2(n-k)-1}.$$
Hence $p^{-1}(x)=p_1^{-1}(x)\times p_2^{-1}(x)$. So we get
$$T_{\mathcal{U}/X}=\pi_1^{*}(T_{\mathcal{U}_1/X})\oplus \pi_2^{*}(T_{\mathcal{U}_2/X}).$$
We mimic the proof of Lemma 5.6 in \cite{D-F-G} to conclude that
$$\pi_1^{*}(T_{\mathcal{U}_1/X})\cong H_k/H_{k-1}\otimes q^{\ast}H_{k-1}^{\vee}, \pi_2^{*}(T_{\mathcal{U}_2/X})\cong (H_{k+1}/H_k)^{\vee}\otimes q^{*}(H_{k+1}^{\bot}/H_{k+1}).$$
Therefore, $T_{\mathcal{U}/X}\cong \big(H_k/H_{k-1}\otimes q^{\ast}H_{k-1}^{\vee}\big)\bigoplus \big((H_{k+1}/H_k)^{\vee}\otimes q^{*}(H_{k+1}^{\bot}/H_{k+1})\big)$ and
$$T_{\mathcal{U}/X}|_{\widetilde{L}}=\mathcal{O}_{\widetilde{L}}(-1)^{\oplus 2n-k-2}.$$
\end{proof}

(3) For $k=n$ (corresponding to the short root $\alpha_{n}$), the odd Orthogonal Grassmannian $B_n/P_n:=OG(n,2n+1)$ parametrizes the $n$-dimensional isotropic subspaces in $V$. In this case, $\mathcal{M}=B_n/P_{n-1}$ and $\mathcal{U}=B_n/P_{n-1,n}$. We have the natural diagram
\begin{align}
\xymatrix{
	\mathcal{U}=B_n/P_{n-1,n}\ar[d]^{p}   \ar[r]^-{q} &\mathcal{M}=B_n/P_{n-1}\\
	X=B_n/P_n
}
\end{align}	
and $\mathcal{M}_x$ is $\mathbb{P}^{n-1}$.
\begin{lemma}\label{bn}
	Let $\widetilde{L}=q^{-1}(l)\subset \mathcal{U}$ for $l\in \mathcal{M}$. For the relative tangent bundle $T_{\mathcal{U}/X}$, we have
	$$T_{\mathcal{U}/X}|_{\widetilde{L}}=\mathcal{O}_{\widetilde{L}}(-2)^{\oplus n-1}.$$
\end{lemma}
\begin{proof}
	It's not hard to check that over $\mathcal{U}$, we have the following exact sequence of vector bundles:
	$$0\rightarrow (H_n/H_{n-1})^{\vee}\rightarrow p^{\ast}H_n^{\vee} \rightarrow q^{\ast}H_{n-1}^{\vee} \rightarrow 0.$$
Restricting it to the $p$-fiber $p^{-1}(x)\cong \mathbb{P}^{n-1}$ is just the Euler sequence on $\mathbb{P}^{n-1}$. In fact, the projection $p$ identifies $\mathcal{U}$ in a canonical fashion with the projective bundle $\mathbb{P}(H_n)$ of $X$. Hence
$T_{\mathcal{U}/X}\cong H_n/H_{n-1}\otimes q^{\ast}H_{n-1}^{\vee}$ and
$$T_{\mathcal{U}/X}|_{\widetilde{L}}=\mathcal{O}_{\widetilde{L}}(-2)^{\oplus n-1}.$$
\end{proof}
\subsubsection{The marked Dynkin diagram $(C_n,k)$}
In this section, we consider the Dynkin diagram $C_n$, which corresponds to the classical Lie group $Sp_{2n}$. Denote by $C_n/P_I:=Sp_{2n}/P_I$ the generalized flag manifold with $I=\{i_1,\ldots,i_s\}\subseteq D$.  Let $V=\mathbb{C}^{2n}$ be a vector space equipped with a nondegenerate skew-symmetric bilinear form $\Omega$. Then $C_n/P_I$ is actually the \emph{Lagrangian flag manifold} $LG(i_1,\ldots,i_s;2n)$, which parametrizing flags
$$V_{i_1}\subset V_{i_2}\subset\cdots\subset V_{i_s}\subset V,$$
where each $V_{i_t}(1\le t\le s)$ is an $i_t$-dimensional isotropic subspace in $V$.

There is universal flag of subbundles
$$0=H_0\subset H_{i_1}\subset \cdots\subset H_{i_s}\subset H_{i_s}^{\bot}\subset\cdots\subset H_{i_1}^{\bot}\subset \mathcal{O}_{C_n/P_I}\times V$$
on $C_n/P_I$, where $H_{i_t}^{\bot}$ is the \emph{$\Omega$-orthogonal complement} of $H_{i_t}$, rank $H_{i_t}=i_t$ and  rank $H_{i_t}^{\bot}=2n-i_t~(1\le t\le s)$.

(1) For $k=1$ (corresponding to the short root $\alpha_{1}$), the Lagrangian Grassmannian $C_n/P_1:=LG(1,2n)$ is just the projective space $\mathbb{P}^{2n-1}$. In this case, $\mathcal{M}=C_n/P_2$ and $\mathcal{U}=C_n/P_{1,2}$. We have the natural diagram
\begin{align}
\xymatrix{
	\mathcal{U}=C_n/P_{1,2}\ar[d]^{p}   \ar[r]^-{q} &\mathcal{M}=C_n/P_2\\
	X=C_n/P_1
}
\end{align}	
and $\mathcal{M}_x$ is $C_{n-1}/P_1$, i.e. the the projective space $\mathbb{P}^{2n-3}$.
\begin{lemma}\label{c1}
	Let $\widetilde{L}=q^{-1}(l)\subset \mathcal{U}$ for $l\in \mathcal{M}$. For the relative tangent bundle $T_{\mathcal{U}/X}$, we have
	$$T_{\mathcal{U}/X}|_{\widetilde{L}}=\mathcal{O}_{\widetilde{L}}(-1)^{\oplus 2n-4}\oplus\mathcal{O}_{\widetilde{L}}(-2).$$
\end{lemma}
\begin{proof}
	It's not hard to check that over $\mathcal{U}$, we have the following exact sequence of vector bundles:
	$$0\rightarrow H_2/H_1\rightarrow p^{*}(H_1^{\bot}/H_1)\rightarrow H_1^{\bot}/H_2\rightarrow 0.$$
	Restricting it to the $p$-fiber $p^{-1}(x)\cong \mathbb{P}^{2n-3}$ is just the Euler sequence on $\mathbb{P}^{2n-3}$. In fact, the projection $p$ identifies $\mathcal{U}$ in a canonical fashion with the projective bundle $\mathbb{P}(H_1^{\bot}/H_1)$ of $X$. Hence
	$T_{\mathcal{U}/X}\cong (H_2/H_1)^{\vee}\otimes H_1^{\bot}/H_2$ and
	$$T_{\mathcal{U}/X}|_{\widetilde{L}}=\mathcal{O}_{\widetilde{L}}(-1)^{\oplus 2n-4}\oplus\mathcal{O}_{\widetilde{L}}(-2) .$$
\end{proof}

(2) For $k~(2\le k\le n-1)$ (corresponding to the short root $\alpha_{k}$), the Lagrangian Grassmannian $C_n/P_k:=LG(k,2n+1)$ parametrizes the $k$-dimensional isotropic subspaces in $V$. In this case $\mathcal{M}=C_n/P_{k-1,k+1}$ and $\mathcal{U}=C_n/P_{k-1,k,k+1}$. We have the natural diagram
\begin{align}
\xymatrix{
	\mathcal{U}=C_n/P_{k-1,k,k+1}\ar[d]^{p}   \ar[r]^-{q} &\mathcal{M}=C_n/P_{k-1,k+1}\\
	X=C_n/P_k
}
\end{align}	
and $\mathcal{M}_x$ is $\mathbb{P}^{k-1}\times C_{n-k}/P_1=\mathbb{P}^{k-1}\times \mathbb{P}^{2(n-k)-1}$.
\begin{lemma}\label{ck}
	Let $\widetilde{L}=q^{-1}(l)\subset \mathcal{U}^k$ for $l\in \mathcal{M}$. For the relative tangent bundle $T_{\mathcal{U}/X}$, we have
	$$T_{\mathcal{U}/X}|_{\widetilde{L}}=\mathcal{O}_{\widetilde{L}}(-1)^{\oplus 2n-k-3}\oplus\mathcal{O}_{\widetilde{L}}(-2).$$
\end{lemma}
\begin{proof}
	It's not hard to check that over $\mathcal{U}$, we have the following two exact sequences of vector bundles:
	\begin{align*}
	0\rightarrow (H_k/H_{k-1})^{\vee}\rightarrow p^{\ast}H_k^{\vee} \rightarrow q^{\ast}H_{k-1}^{\vee} \rightarrow 0,
	\end{align*}
	\begin{align*}
	0\rightarrow H_{k+1}/H_k\rightarrow p^{*}(H_k^{\bot}/H_k)\rightarrow H_k^{\bot}/H_{k+1}\rightarrow 0.
	\end{align*}
	We will consider the following diagram
	\begin{align}
	\xymatrix{
		\mathcal{U}_1=C_n/P_{k-1,k}  \ar[rrdd]_{p_1} &&\mathcal{U}=C_n/P_{k-1,k,k+1}\ar[ll]_{\pi_1}\ar[rr]^{\pi_2}\ar[dd]^{p}  && \mathcal{U}_2=C_n/P_{k,k+1}  \ar[lldd]^{p_2} \\
		\\
		&&X.
	}
	\end{align}
	All the morphisms in the above diagram are projections. For any $x\in X$, the $p_1$-fiber $p_1^{-1}(x)$ is isomorphic to $\mathbb{P}^{k-1}$ and the $p_2$-fiber $p_2^{-1}(x)$ is isomorphic to $\mathbb{P}^{2(n-k)-1}$. Note that for $x\in X$, the projection $q$ induces an isomorphism
	$$q|_{p^{-1}(x)}:p^{-1}(x)\rightarrow \mathcal{M}_x=\mathbb{P}^{k-1}\times \mathbb{P}^{2(n-k)-1}.$$
	Hence $p^{-1}(x)=p_1^{-1}(x)\times p_2^{-1}(x)$. So we get
	$$T_{\mathcal{U}/X}=\pi_1^{*}(T_{\mathcal{U}_1/X})\oplus \pi_2^{*}(T_{\mathcal{U}_2/X}).$$
	We mimic the proof of Lemma 5.6 in \cite{D-F-G} to conclude that
	$$\pi_1^{*}(T_{\mathcal{U}_1/X})\cong H_k/H_{k-1}\otimes q^{\ast}H_{k-1}^{\vee}, \pi_2^{*}(T_{\mathcal{U}_2/X})\cong (H_{k+1}/H_k)^{\vee}\otimes H_k^{\bot}/H_{k+1}.$$
	Therefore, $T_{\mathcal{U}/X}\cong \big(H_k/H_{k-1}\otimes q^{\ast}H_{k-1}^{\vee}\big)\bigoplus \big(H_{k+1}/H_k)^{\vee}\otimes H_k^{\bot}/H_{k+1}\big)$ and
	$$T_{\mathcal{U}/X}|_{\widetilde{L}}=\mathcal{O}_{\widetilde{L}}(-1)^{\oplus 2n-k-3}\oplus\mathcal{O}_{\widetilde{L}}(-2).$$
\end{proof}

(3) For $k=n$ (corresponding to the long root $\alpha_{n}$), the Lagrangian Grassmannian $C_n/P_n:=LG(n,2n)$ parametrizes the $n$-dimensional isotropic subspaces in $V$. In this case $\mathcal{M}=C_n/P_{n-1}$ and $\mathcal{U}=C_n/P_{n-1,n}$. We have the natural diagram
\begin{align}
\xymatrix{
	\mathcal{U}=C_n/P_{n-1,n}\ar[d]^{p}   \ar[r]^-{q} &\mathcal{M}=C_n/P_{n-1}\\
	X=C_n/P_n
}
\end{align}	
and $\mathcal{M}_x$ is $\mathbb{P}^{n-1}$.
\begin{lemma}\label{cn}
	Let $\widetilde{L}=q^{-1}(l)\subset \mathcal{U}$ for $l\in \mathcal{M}$. For the relative tangent bundle $T_{\mathcal{U}/X}$, we have
	$$T_{\mathcal{U}/X}|_{\widetilde{L}}=\mathcal{O}_{\widetilde{L}}(-1)^{\oplus n-1}.$$
\end{lemma}
\begin{proof}
	It's not hard to check that over $\mathcal{U}$, we have the following exact sequence of vector bundles:
	$$0\rightarrow (H_n/H_{n-1})^{\vee}\rightarrow p^{\ast}H_n^{\vee} \rightarrow q^{\ast}H_{n-1}^{\vee} \rightarrow 0.$$
	Restricting it to the $p$-fiber $p^{-1}(x)\cong \mathbb{P}^{n-1}$ is just the Euler sequence on $\mathbb{P}^{n-1}$. In fact, the projection $p$ identifies $\mathcal{U}$ in a canonical fashion with the projective bundle $\mathbb{P}(H_n)$ of $X$. Hence
	$T_{\mathcal{U}/X}\cong H_n/H_{n-1}\otimes q^{\ast}H_{n-1}^{\vee}$ and
	$$T_{\mathcal{U}/X}|_{\widetilde{L}}=\mathcal{O}_{\widetilde{L}}(-1)^{\oplus n-1}.$$
\end{proof}
\subsubsection{The marked Dynkin diagram $(D_n,k)$}
In this section, we consider the Dynkin diagram $D_n$, it corresponds to the classical Lie group $SO_{2n}$. Denote by $D_n/P_I:=SO_{2n}/P_I$ the generalized flag manifold with $I=\{i_1,\ldots,i_s\}\subseteq D$.  Let $V=\mathbb{C}^{2n}$ be a vector space equipped with a nondegenerate symmetric bilinear form $\mathcal{Q}$. Then

Case 1. $n-1,n\in I$. $D_n/P_I$ is the \emph{even Orthogonal flag manifold} $OG(i_1,\ldots,i_{s-2},n-1,n;2n)$ which parametrizes two families of flags
$$V_{i_1}\subset\cdots\subset V_{i_{s-2}}\subset V_{n-1}\subset V~ \text{and}~ V_{i_1}\subset\cdots\subset V_{i_{s-2}}\subset V_{n-1}'\subset V$$
where each $V_{i_t}(1\le t\le s-2)$ is an $i_t$-dimensional isotropic subspace in $V$, $V_{n-1}$ and $V_{n-1}'$ are the $(n-1)$-dimensional isotropic subspaces in $V$.

There are universal flags of subbundles
\begin{align*}
0=H_0\subset H_{i_1}\subset \cdots\subset H_{n-1}\subset \mathcal{O}_{D_n/P_I}\times V,\\
0=H_0\subset H_{i_1}\subset \cdots\subset H_{n-1}'\subset \mathcal{O}_{D_n/P_I}\times V
\end{align*}
on $D_n/P_I$, where rank $H_{i_t}=i_t~(1\le t\le s-2)$ and $\text{rank} ~H_{n-1}=\text{rank} ~H_{n-1}'=n-1$.

Case 2. $n-1\in I,n\notin I$ or $n-1\notin I, n\in I$. $D_n/P_I$ is one of the two irreducible components of the \emph{even Orthogonal flag manifold} $OG(i_1,\ldots,i_{s-1},n;2n)~(i_s=n)$, which parametrizes flag
$$V_{i_1}\subset V_{i_2}\subset\cdots\subset V_n\subset V,$$
where each $V_{i_t}(1\le t\le s)$ is an $i_t$-dimensional isotropic subspace in $V$.

There is a universal flag of subbundles
$$0=H_0\subset H_{i_1}\subset \cdots\subset H_n\subset \mathcal{O}_{D_n/P_I}\times V$$
on $D_n/P_I$, where rank $H_{i_t}=i_t~(1\le t\le s)$.

Case 3. $n-1,n\notin I$. $D_n/P_I$ is actually the \emph{even Orthogonal flag manifold} $OG(i_1,\ldots,i_{s-1},i_s;2n)$, which parametrizes flag
$$V_{i_1}\subset V_{i_2}\subset\cdots\subset V_{i_s}\subset V,$$
where each $V_{i_t}(1\le t\le s)$ is an $i_t$-dimensional isotropic subspace in $V$.

There is a universal flag of subbundles
$$0=H_0\subset H_{i_1}\subset \cdots\subset H_{i_s}\subset H_{i_s}^{\bot}\subset\cdots\subset H_{i_1}^{\bot}\subset \mathcal{O}_{D_n/P_I}\times V$$
on $D_n/P_I$, where $H_{i_t}^{\bot}$ is the \emph{$\mathcal{Q}$-orthogonal complement} of $H_{i_t}$, rank $H_{i_t}=i_t$ and rank $H_{i_t}^{\bot}=2n-i_t~(1\le t\le s)$.

(1) For $k~(1\le k\le n-3)$, the even Orthogonal Grassmannian $D_n/P_k:=OG(k,2n)$ parametrizes the $k$-dimensional isotropic subspaces in $V$. In this case, $\mathcal{M}=D_n/P_{k-1,k+1}$ and $\mathcal{U}=D_n/P_{k-1,k,k+1}$, i.e. we have the natural diagram:
\begin{align}
\xymatrix{
	\mathcal{U}=D_n/P_{k-1,k,k+1}\ar[d]^{p}   \ar[r]^-{q} &\mathcal{M}=D_n/P_{k-1,k+1}\\
	X=D_n/P_k
}
\end{align}	
and $\mathcal{M}_x$ is $\mathbb{P}^{k-1}\times D_{n-k}/P_1=\mathbb{P}^{k-1}\times Q^{2(n-k-1)}$.
\begin{lemma}\label{dk}
	Let $\widetilde{L}=q^{-1}(l)\subset \mathcal{U}$ for $l\in \mathcal{M}$. For the relative tangent bundle $T_{\mathcal{U}/X}$, we have
	$$T_{\mathcal{U}/X}|_{\widetilde{L}}=\mathcal{O}_{\widetilde{L}}(-1)^{\oplus 2n-k-3}.$$
\end{lemma}
\begin{proof}
	It's not hard to check that over $\mathcal{U}$, we have the following two exact sequences of vector bundles:
	\begin{align*}
	0\rightarrow (H_k/H_{k-1})^{\vee}\rightarrow p^{\ast}H_k^{\vee} \rightarrow q^{\ast}H_{k-1}^{\vee} \rightarrow 0,
	\end{align*}
	\begin{align*}
	0\rightarrow H_{k+1}/H_k\rightarrow H_{k+1}^{\bot}/H_k\rightarrow q^{*}(H_{k+1}^{\bot}/H_{k+1})\rightarrow 0.
	\end{align*}
	We will consider the following diagram
	\begin{align}
	\xymatrix{
		\mathcal{U}_1=D_n/P_{k-1,k}  \ar[rrdd]_{p_1} &&\mathcal{U}=D_n/P_{k-1,k,k+1}\ar[ll]_{\pi_1}\ar[rr]^{\pi_2}\ar[dd]^{p}  && \mathcal{U}_2=D_n/P_{k,k+1}  \ar[lldd]^{p_2} \\
		\\
		&&X.
	}
	\end{align}
	All the morphisms in the above diagram are projections. For any $x\in X$, the $p_1$-fiber $p_1^{-1}(x)$ is isomorphic to $\mathbb{P}^{k-1}$ and the $p_2$-fiber $p_2^{-1}(x)$ is isomorphic to $Q^{2(n-k-1)}$. Note that for $x\in X$, the projection $q$ induces an isomorphism
	$$q|p^{-1}(x):p^{-1}(x)\rightarrow \mathcal{M}_x=\mathbb{P}^{k-1}\times Q^{2(n-k-1)}.$$
	Hence $p^{-1}(x)=p_1^{-1}(x)\times p_2^{-1}(x)$. So we get
	$$T_{\mathcal{U}/X}=\pi_1^{*}(T_{\mathcal{U}_1/X})\oplus \pi_2^{*}(T_{\mathcal{U}_2/X}).$$
	We mimic the proof of Lemma 5.6 in \cite{D-F-G} to conclude that
	$$\pi_1^{*}(T_{\mathcal{U}_1/X})\cong H_k/H_{k-1}\otimes q^{\ast}H_{k-1}^{\vee}, \pi_2^{*}(T_{\mathcal{U}_2/X})\cong (H_{k+1}/H_k)^{\vee}\otimes q^{*}(H_{k+1}^{\bot}/H_{k+1}).$$
	Therefore, $T_{\mathcal{U}/X}\cong \big(H_k/H_{k-1}\otimes q^{\ast}H_{k-1}^{\vee}\big)\bigoplus \big((H_{k+1}/H_k)^{\vee}\otimes q^{*}(H_{k+1}^{\bot}/H_{k+1})\big)$ and
	$$T_{\mathcal{U}/X}|_{\widetilde{L}}=\mathcal{O}_{\widetilde{L}}(-1)^{\oplus 2n-k-3}.$$
\end{proof}

(2) For $k=n-2$, the even Orthogonal Grassmannian $D_n/P_{n-2}:=OG(n-2,2n)$ parametrizes the $n-2$-dimensional isotropic subspaces in $V$. In this case $\mathcal{M}=D_n/P_{n-3,n-1,n}$ and $\mathcal{U}=D_n/P_{n-3,n-2,n-1,n}$. We have the natural diagram
\begin{align}
\xymatrix{
	\mathcal{U}=D_n/P_{n-3,n-2,n-1,n}\ar[d]^{p}   \ar[r]^-{q} &\mathcal{M}=D_n/P_{n-3,n-1,n}\\
	X=D_n/P_{n-2}
}
\end{align}	
and $\mathcal{M}_x$ is $\mathbb{P}^{n-3}\times \mathbb{P}^{1}\times \mathbb{P}^{1}$.
\begin{lemma}\label{dn-1}
	Let $\widetilde{L}=q^{-1}(l)\subset \mathcal{U}$ for $l\in \mathcal{M}$. For the relative tangent bundle $T_{\mathcal{U}/X}$, we have
	$$T_{\mathcal{U}/X}|_{\widetilde{L}}=\mathcal{O}_{\widetilde{L}}(-1)^{\oplus n-1}.$$
\end{lemma}
\begin{proof}
	It's not hard to check that over $\mathcal{U}$, we have the following three exact sequences of vector bundles:
	\begin{align*}
	0\rightarrow (H_{n-2}/H_{n-3})^{\vee}\rightarrow p^{\ast}H_{n-2}^{\vee} \rightarrow q^{\ast}H_{n-3}^{\vee} \rightarrow 0,
	\end{align*}
	\begin{align*}
	0\rightarrow H_{n-1}/H_{n-2}\rightarrow \mathcal{O}_{\mathcal{U}}^{\oplus n}/H_{n-2}\rightarrow q^{*}(\mathcal{O}_{\mathcal{M}}^{\oplus n}/H_{n-1})\rightarrow 0.
	\end{align*}
	\begin{align*}
	0\rightarrow H_{n-1}'/H_{n-2}\rightarrow \mathcal{O}_{\mathcal{U}}^{\oplus n}/H_{n-2}\rightarrow q^{*}(\mathcal{O}_{\mathcal{M}}^{\oplus n}/H_{n-1}')\rightarrow 0.
	\end{align*}
	We will consider the following diagram
	\begin{align}
	\xymatrix{
		\mathcal{U}_1=D_n/P_{n-3,n}  \ar[rrdd]_{p_1} &&	\mathcal{U}=D_n/P_{n-3,n-2,n-1,n}\ar[ll]_{\pi_1}\ar[rr]^{\pi_2}\ar[dd]^{p}  && \mathcal{U}_2=D_n/P_{n-2,n-1,n}  \ar[lldd]^{p_2} \\
		\\
		&&X.
	}
	\end{align}
	All the morphisms in the above diagram are projections. For any $x\in X$, the $p_1$-fiber $p_1^{-1}(x)$ is isomorphic to $\mathbb{P}^{n-3}$ and the $p_2$-fiber $p_2^{-1}(x)$ is isomorphic to $\mathbb{P}^1\times\mathbb{P}^1$. Note that for $x\in X$, the projection $q$ induces an isomorphism
	$$q|_{p^{-1}(x)}:p^{-1}(x)\rightarrow \mathcal{M}_x=\mathbb{P}^{n-3}\times\mathbb{P}^1\times\mathbb{P}^1.$$
	Hence $p^{-1}(x)=p_1^{-1}(x)\times p_2^{-1}(x)$. So we get
	$$T_{\mathcal{U}/X}=\pi_1^{*}(T_{\mathcal{U}_1/X})\oplus \pi_2^{*}(T_{\mathcal{U}_2/X}).$$
	We mimic the proof of Lemma 5.6 in \cite{D-F-G} to conclude that
	$$\pi_1^{*}(T_{\mathcal{U}_1/X})\cong H_{n-2}/H_{n-3}\otimes q^{\ast}H_{n-3}^{\vee}$$
	and
	$$ \pi_2^{*}(T_{\mathcal{U}_2/X})\cong (H_{n-1}/H_{n-2})^{\vee}\otimes q^{*}(\mathcal{O}_{\mathcal{M}}^{\oplus n}/H_{n-1})\oplus (H_{n-1}'/H_{n-2})^{\vee}\otimes q^{*}(\mathcal{O}_{\mathcal{M}}^{\oplus n}/H_{n-1}').$$
	Therefore,
	\begin{align*}
	&T_{\mathcal{U}/X}\cong
\big(H_{n-2}/H_{n-3}\otimes q^{\ast}H_{n-3}^{\vee}\big)\bigoplus \\
&\big((H_{n-1}/H_{n-2})^{\vee}\otimes q^{*}(\mathcal{O}_{\mathcal{M}}^{\oplus n}/H_{n-1})\big)\bigoplus  \big((H_{n-1}'/H_{n-2})^{\vee}\otimes q^{*}(\mathcal{O}_{\mathcal{M}}^{\oplus n}/H_{n-1}')\big)
	\end{align*}
	 and
	$T_{\mathcal{U}/X}|_{\widetilde{L}}=\mathcal{O}_{\widetilde{L}}(-1)^{\oplus n-1}.$
\end{proof}

(3) For $k=n-1$ or $n$, $D_n/P_k$ is one of the two irreducible components of the even Orthogonal Grassmannian $OG(n,2n)$. In this case $\mathcal{M}=D_n/P_{n-2}$ and $\mathcal{U}=D_n/P_{n-2,k}$. We have the natural diagram
\begin{align}
\xymatrix{
\mathcal{U}=D_n/P_{n-2,k}\ar[d]^{p}   \ar[r]^-{q} &\mathcal{M}=D_n/P_{n-2}\\
	X=D_n/P_k
}
\end{align}	
and $\mathcal{M}_x$ is Grassmannian $G(2,n)$.
\begin{lemma}\label{dn}
	Let $\widetilde{L}=q^{-1}(l)\subset \mathcal{U}$ for $l\in \mathcal{M}$. For the relative tangent bundle $T_{\mathcal{U}/X}$, we have
	$$T_{\mathcal{U}/X}|_{\widetilde{L}}=\mathcal{O}_{\widetilde{L}}(-1)^{\oplus 2(n-2)}.$$
\end{lemma}
\begin{proof}
	It's not hard to check that over $\mathcal{U}$, we have the following exact sequence of vector bundles
$$0\rightarrow (H_n/H_{n-2})^{\vee}\rightarrow p^{\ast}H_n^{\vee} \rightarrow q^{\ast}H_{n-2}^{\vee} \rightarrow 0.$$
Restricting it to the $p$-fiber $p^{-1}(x)\cong G(2,n)$ is just the Euler sequence on $G(2,n)$. In fact, the projection $p$ identifies $\mathcal{U}$ in a canonical fashion with the Grassmannian bundle $G(2,H_n)$ of $X$. Hence
$T_{\mathcal{U}/X}\cong H_n/H_{n-2}\otimes q^{\ast}H_{n-2}^{\vee}$ and
$$T_{\mathcal{U}/X}|_{\widetilde{L}}=\mathcal{O}_{\widetilde{L}}(-1)^{\oplus 2(n-2)}.$$
\end{proof}
\subsection{The exceptional simple Lie algebra}
We say $G$ is of exceptional type if its Dynkin diagram is of type $E_n~(n=6,7,8)$, $F_4$ or $G_2$. Due to the complexity of the geometry of $X=G/P_k$, it is difficult to write down the specific form of the relative tangent bundle $T_{\mathcal{U}/X}$, so here we use the method in \cite{M-O-C} to calculate the splitting type of $T_{\mathcal{U}/X}|_{q^{-1}(l)}$.

When $G$ is simple, the fundamental root system will be written as in the standard reference \cite{R-C}. In the case of $E_n~(n=6,7,8)$ type, the mark of the nodes of the Dynkin diagam in the reference is different from our paper, so is the expression of the fundamental root system. So we write down the fundamental root system corresponding to the diagram of $E_n~(n=6,7,8)$-type in our paper as follow:

Let $V$ be a real vector space with dim $V=8$ with orthogonal basis $e_i, i=1,\ldots,8$. Then the vectors
\begin{align*}
\alpha_1=&-\frac{1}{2}(e_1+\cdots+e_8),~\alpha_2=e_6-e_7,\\
\alpha_3=&e_6+e_7,~\alpha_4=e_5-e_6,\\
\alpha_5=&e_4-e_5, ~\alpha_6=e_3-e_4,\\
\alpha_7=&e_2-e_3,~\alpha_8=e_1-e_2
\end{align*}
form a fundamental root system of type $E_8$. Since $E_6,E_7$ can be identified canonically with subsystem of $E_8$, so $\alpha_1,\ldots,\alpha_n$ form a fundamental root system of type $E_n (n=6,7)$.

 It's well known that homogeneous vector bundles especially the tangent bundle on a rational homogeneous space $G/P$ are determined by representations of the Lie algebra $\mathfrak{p}$. Restricting the relative tangent bundle $T_{\mathcal{U}/X}$ to a line $\widetilde{L}=q^{-1}(l)$, we will get a flag bundle $H/P_{N(k)}$ and $D(H)=D\backslash {k}$ with tag, where the tag is the set of intersection numbers of a minimal section of  the associated $H/P_{N(k)}$-flag bundle with the relative canonical divisors (see \cite{O-C-W} Proposition 3.17). We refer to \cite{O-C-W, M-O-C-W} for a complete account on the tag of a flag bundle. Assume $H/P_{N(k)}:=G_1/P_{I_1}\times G_2/P_{I_2}\times \cdots \times G_m/P_{I_m}$ ($m\le 3$) where $G_i$ is simple Lie group. Calculating weights that give the tangent bundle of $G_i/P_{I_i}~(1\le i\le m)$, which are all the positive roots of the Lie algebra of $G_i$ that contain the root adjacent to $\alpha_k$. Putting together the above tag and weights, we will get the splitting type of $T_{\mathcal{U}/X}|_{\widetilde{L}}$. We refer to (\cite{M-O-C-W}, page 5-7) for a complete account. By calculation, we have the tables (Table \ref{E6}, \ref{E7}, \ref{E8}, \ref{FG}) in Appendix.

Professor L. E. Sol{\'{a}} Conde told us the following example for calculating the relative tangent bundle restricting to an isotropic line by personal correspondance.

 \begin{ex}
\emph{ Let $X=E_7/P_7$ and $\mathcal{U}=E_7/P_{6,7}$. Then $H/P_{N(7)}=E_6/P_{6}$. Let $E_7/P_I:=E_7/P_{\{1,2,3,4,5,6,7\}}$.
 The isotropic lines in $E_7/P_{i}$ are the image of the fibers of the $\mathbb{P}^1$-bundle
 \[p^i:=E_7/P_I\rightarrow E_7/P_{I\backslash i},~ (i=1,2, \ldots 7)\]
  into $E_7/P_{7}$ via the natural map $$\pi^i: E_7/P_I\rightarrow E_7/P_{i}.$$
Let $-K_i~ (i=1,2, \ldots 7)$ be the relative anticanonical bundles of $p^i$ and $C_i$ be a fiber of $p^i$. When mapped to the varieties of Picard number one $E_7/P_i$'s, $C_i$'s are the isotropic lines, where $1\le i\le 7$. The $C_7$ can be regarded as a minimal section of the fibration morphism $E_7/P_{6,7}\rightarrow E_7/P_7 $ when it restricts to $\pi^7(C_7)$.
By \cite{O-C-W-W} Proposition 2.13, the matrix $(-K_iC_j)$ is the Cartan matrix of $E_7$. The tag is exactly $$(K_1C_7,...,K_6C_7)=(0,0,0,0,0,1),$$ which is the $7$-th column of the Cartan matrix without the last term and with the signs changed.
The weights are coefficients of all the positive roots of the Lie algebra of $E_6$ in term of the linear combinations of the fundamental roots that contain the $6$-th root as a summand; that is the roots of the form $m_1\alpha_1+\cdots+m_6\alpha_6$ with $m_i\geq 0,~ m_6>0$, where $\alpha_i$'s are the fundamental roots ($i=1,\ldots, 6$) (See Table \ref{E7} in Appendix for the values of $m_i$). So
\[T_{\mathcal{U}/X}|_{\widetilde{L}}=\mathcal{O}_{\widetilde{L}}(-1)^{\oplus16}.\]
}
 \end{ex}

 Combining Table \ref{E6}, \ref{E7}, \ref{E8}, \ref{FG} we can get the following Proposition.
\begin{prop}\label{fk}
	Suppose $G$ to be of exceptional type and $X=G/P_k$. Let $\widetilde{L}=q^{-1}(l)\subset \mathcal{U}$ for $l\in \mathcal{M}$.
	When $X=F_4/P_3,F_4/P_4,G_2/P_1$, $T_{\mathcal{U}/X}|_{\widetilde{L}}$ has the following form  respectively:
	$$\mathcal{O}_{\widetilde{L}}(-2)^{\oplus 2}\oplus\mathcal{O}_{\widetilde{L}}(-1),\mathcal{O}_{\widetilde{L}}(-2)^{\oplus 3}\oplus\mathcal{O}_{\widetilde{L}}(-1)^{\oplus 3}, \mathcal{O}_{\widetilde{L}}(-3).$$
Otherwise $$T_{\mathcal{U}/X}|_{\widetilde{L}}=\mathcal{O}_{\widetilde{L}}(-1)^{N}, ~\text{where}~N=dim~ \mathcal{U}-dim X.$$
\end{prop}
Summing up, we have obtained the following
\begin{prop}\label{tangent}
Let $G$ be a simple Lie group with the Dynkin diagram $\mathcal{D}$ and $\alpha_k$ is a long root of $\mathcal{D}$. Denote by $X$ the generalized Grassmannian by marking on $\mathcal{D}$ the node $k$. Then for the relative tangent bundle $T_{\mathcal{U}/X}$ (Notations as Section 4), we have
$$T_{\mathcal{U}/X}|_{\widetilde{L}}=\mathcal{O}_{\widetilde{L}}(-1)^{N},~\text{where}~N=dim~ \mathcal{U}-dim X.$$
\end{prop}
\begin{proof}
It is obviously from Lemma \ref{b1},\ref{bk},\ref{cn},\ref{dk},\ref{dn-1},\ref{dn} and Proposition \ref{fk}.
\end{proof}
\subsection{The generalization of the Grauert-M$\ddot{\text{u}}$lich-Barth theorem}
The construction of subsheaves in holomorphic vector bundles plays an important role in the proof of the generalized Grauert-M$\ddot{\text{u}}$lich-Barth theorem. A generalization of the Grauert-M$\ddot{\text{u}}$lich-Barth theorem to normal projective varieties in characteristic zero is proved (see \cite{H-L} Theorem 3.1.2). Since the theorem in the book is for any normal projective variety, the bound is pretty coarse. In this section, we will find the explicit bound for rational homogeneous spaces.

The following Descent Lemma provides a way for us to prove the existence of subsheaves.
\begin{lemma}(Descent Lemma \cite{O-S-S})\label{Descent}
	Let $X$, $Y$ be nonsingular varieties over $k$, $f:X\rightarrow Y$ be a surjective submersion with connected fibers and $E$ be an algebraic $r$-bundle over $Y$. Let $\widetilde{K}\subset f^{\ast}E$ be a subbundle of rank $t$ in $f^{\ast}E$ and $\widetilde{Q}=f^{\ast}E/\widetilde{K}$ be its quotient. If \[
	Hom(T_{X/Y},\mathcal{H}om(\widetilde{K},\widetilde{Q}))=0,
	\]
	then $\widetilde{K}$ is the form $\widetilde{K}=f^{\ast}K$ for some algebraic subbundle $K\subset E$ of rank $t$.
\end{lemma}

Follow the previous notations. Let $G$ be a simple Lie group with the Dynkin diagram $\mathcal{D}$ and $\alpha_k$ is a root of $\mathcal{D}$. Let's consider the standard diagram associated to $X$
\begin{align}
\xymatrix{
	\mathcal{U}\ar[d]^{p}   \ar[r]^-{q} & \mathcal{M}\\
	X.
}
\end{align}	
\begin{thm}
 Let $X$ be a generalized Grassmannian by marking on $\mathcal{D}$ the node $k$ and $\alpha_k$ be a long root of $\mathcal{D}$. Let $E$ be a holomorphic $r$-bundle over $X$ of type $\underline{a}_E=(a_1,\ldots,a_r), a_1\geq\cdots\geq a_r$. If for some $t<r$,
$$a_t-a_{t+1}\geq 2~~for ~some ~t<r,$$
then there is a normal subsheaf $K\subset E$ of rank $t$ with the following properties: over the open set $V_E=p(q^{-1}(U_E))\subset X$, where $U_E$ is an open set in $\mathcal{M}$, the sheaf $K$ is a subbundle of $E$, which on the line $L\subset X$ given by $l\in U_E$ has the form
$$K|_L\cong\oplus_{j=1}^{t}\mathcal{O}_L(a_j).$$
\end{thm}
\begin{proof}
The Proposition \ref{tangent} and the Lemma \ref{Descent} play important roles in our proof and our proof applies almost verbatim the proof of Theorem 2.1.4 given in \cite{O-S-S}.
\end{proof}
This theorem has far reaching consequences. We give first a series of immediate deductions.
\begin{cor}
Let $X$ be a generalized Grassmannian with long root $\alpha_k$. For a semistable $r$-bundle $E$ over $X$ of type $\underline{a}_E=(a_1,\ldots,a_r), a_1\geq\cdots\geq a_r$. we have
$$a_i-a_{i+1}\le 1~~for ~i=1,\ldots,r-1.$$
\end{cor}
\begin{proof}
If for some $t<r$, we had $a_t-a_{t+1}\geq 2$, then we could find a normal subsheaf $K\subset E$ which is of the form
$$K|_L\cong\oplus_{j=1}^{t}\mathcal{O}_L(a_j)$$
over the general line $L\subset X$. Then we would have $\mu(E)<\mu(K)$ contrary to hypothesis.
\end{proof}
In particular, we get the generalization theorem of Grauert-M$\ddot{\text{u}}$lich-Barth :
\begin{cor}
Let $X$ be a generalized Grassmannian with long root $\alpha_k$. The splitting type of a semistable normalized 2-bundle $E$ over generalized Grassmannian $X$ is
\[
a_E=
\left\{
\begin{array}{ll}
(0,0)& if~c_1(E)=0\\
(0,-1)& if~c_1(E)=-1.
\end{array}
\right.\]
\end{cor}
\begin{cor}
Let $X$ be a generalized Grassmannian with long root $\alpha_k$. For a uniform $r$-bundle $E~(r=\varsigma(X)+1)$ over $X$ of type (see Section 3 for the notation $\varsigma(X)$)
$$\underline{a}_E=(a_1,\ldots,a_r),~ a_1\geq\cdots\geq a_r,$$
which does not split, we have
$$a_i-a_{i+1}\le 1~~for ~i=1,\ldots,r-1.$$
\end{cor}
\begin{proof}
If for some $t<r$, we had $a_t-a_{t+1}\geq 2$, then we could find a uniform subbundle $K\subset E$ of type which is of $\underline{a}_K=(a_1,\ldots,a_t)$ (because $V_E=X$). Then quotient bundle $Q=E/K$ would be uniform of type $(a_{s+1},\ldots,a_r)$. According to Theorem \ref{GGS} the bundle $K$ and $Q$ must be direct sums of line bundles. The exact sequence
$$0\rightarrow K\rightarrow E\rightarrow Q\rightarrow 0$$
would therefore split and hence $E$ would be a direct sum of line bundles contrary to hypothesis.
\end{proof}

When $\alpha_k$ is a short root of $\mathcal{D}$, $\mathcal{M}$ is not the variety of lines on $X$, but only a closed $G$-orbit. Therefore we're just going to think about the splitting type and semistability of vector bundles with respect to $\mathcal{M}$.
\begin{thm}
 Let $X$ be a generalized Grassmannian by marking on $\mathcal{D}$ the node $k$ and $\alpha_k$ be a short root of $\mathcal{D}$. Let $E$ be a holomorphic $r$-bundle over $X$ of type $\underline{a}_E=(a_1,\ldots,a_r), a_1\geq\cdots\geq a_r$ with respect to $\mathcal{M}$. If for some $t<r$,
 \[
 a_t-a_{t+1}\geq
 \left\{
 \begin{array}{ll}
 3, & if ~\mathcal{D}\neq G_2\\
 4, & if ~\mathcal{D}= G_2,
 \end{array}
 \right.\]
	then there is a normal subsheaf $K\subset E$ of rank $t$ with the following properties: over the open set $V_E=p(q^{-1}(U_E))\subset X$,  where $U_E$ is an open set in $\mathcal{M}$, the sheaf $K$ is a subbundle of $E$, which on the line $L\subset X$ given by $l\in U_E$ has the form
	$$K|_L\cong\oplus_{j=1}^{t}\mathcal{O}_L(a_j).$$
\end{thm}
\begin{proof}
 The Lemma \ref{Descent}, \ref{bn}, \ref{c1}, \ref{ck} and Proposition \ref{fk} play important roles in our proof and our proof applies almost verbatim the proof of Theorem 2.4 given in \cite{O-S-S}.
\end{proof}
Similarly, we have the following corollary.
\begin{cor}
Let $X$ be a generalized Grassmannian with short root $\alpha_k$.	For a semistable $r$-bundle $E$ over $X$ of type $\underline{a}_E=(a_1,\ldots,a_r), a_1\geq\cdots\geq a_r$ with respect to $\mathcal{M}$. we have
	$$a_i-a_{i+1}\le 3~~for ~i=1,\ldots,r-1.$$
	In particular, if $\mathcal{D}\neq G_2$, $a_i-a_{i+1}\le 2~~for ~i=1,\ldots,r-1.$
\end{cor}

From now, Let \[X=G/P\simeq G_1/P_{I_1}\times G_2/P_{I_2}\times \cdots \times G_m/P_{I_m}, \] where $G_i$ is a simple Lie group with \emph{Dynkin diagram} $\mathcal{D}_i$ whose set of nodes is $D_i$ and $P_{I_i}$ is a parabolic subgroup of $G_i$ corresponding to $I_i\subset D_i$. We set $F(I_i):=G_i/P_{I_i}$ by marking on the Dynkin diagram $\mathcal{D}_i$ of $G_i$ the nodes corresponding to $I_i$. Let $\delta_i$ be a node in $\mathcal{D}_i$ and $N(\delta_i)$ be the set of nodes in $\mathcal{D}_i$ that are connected to $\delta_i$.

If $\delta_i\in I_i$, we call $$\mathcal{M}_i^{\delta_i^c}:=G_i/P_i^{\delta_i^c}\times \widehat{G_i/P_{I_i}}~ (1\le i\le m),$$ the \emph{$i$-th special family} of lines of class $\check{\delta}_i$, where $P_i^{\delta_i^c}:=P_{I_i\backslash \delta_i\cup N(\delta_i)}$ and $\widehat{G_i/P_{I_i}}$ is $G_1/P_{I_1}\times G_2/P_{I_2}\times \cdots \times G_m/P_{I_m}$ by deleting $i$-th term $G_i/P_{I_i}$. Denote by
$$\mathcal{U}_i^{\delta_i^c}:=G_i/P_{I_i\cup N(\delta_i)}\times \widehat{G_i/P_{I_i}}$$ the \emph{$i$-th universal family} of class $\check{\delta}_i$, which has a natural $\mathbb{P}^1$-bundle structure over $\mathcal{M}_i^{\delta_i^c}$.

We separate our discussion into two cases:

Case I: $N(\delta_i)\subseteq I_i$, then $\mathcal{U}_i^{\delta_i^c}=X$ and we have the natural projection $X\rightarrow \mathcal{M}_i^{\delta_i^c}$;

Case II: $N(\delta_i)\nsubseteq I_i$, then we have the standard diagram:
\begin{align}
	\xymatrix{
		\mathcal{U}_i^{\delta_i^c}\ar[d]^{q_1}   \ar[r]^-{q_2} & \mathcal{M}_i^{\delta_i^c} \\
		X.
	}
\end{align}


Notice that for $x\in X$, $\mathcal{M}^{\delta_i^c}_x={q_2}({q_1}^{-1}(x))$ coincides with $H/P_{N(\delta_i)}$ where $D(H)$ is the components of $(\overline{D_i\backslash I_i})\backslash \delta_i$ containing an element of $N(\delta_i)$ by Theorem \ref{lines2}.

Let $\mathcal{G}^{\delta_i}$ be a generalized Grassmannian whose Dynkin diagram $\mathcal{D}^{\delta_i}$ is the maximal sub-diagram of $(\mathcal{D}_i,I_i)$ with the only marked point $\delta_i$.
Let's consider the standard diagram associated to $\mathcal{G}^{\delta_i}$
\begin{align}
\xymatrix{
	\mathcal{U}\ar[d]^{p}   \ar[r]^-{q} & \mathcal{M}\\
	\mathcal{G}^{\delta_i}.
}
\end{align}	

It's not hard to see that $\mathcal{M}^{\delta_i^c}_x$ is isomorphic to $\mathcal{M}_y=q(p^{-1}(y))$ for every $y\in \mathcal{G}^{\delta_i}$ and we have the following Lemma.
\begin{lemma}\label{tangent3}
		Let $\widetilde{L}=q_2^{-1}(l)\subset \mathcal{U}_i^{\delta_i^c}$ for $l\in \mathcal{M}_i^{\delta_i^c}$. For the relative tangent bundle $T_{\mathcal{U}_i^{\delta_i^c}/X}$, we have
	$$T_{\mathcal{U}_i^{\delta_i^c}/X}|_{\widetilde{L}}=T_{\mathcal{U}/\mathcal{G}^{\delta_i}}|_{\widetilde{L}}.$$
\end{lemma}

If $\delta_i$ is an exposed short root of $(\mathcal{D}_i,I_i)$, then $\delta_i$ is a short root of $\mathcal{D}^{\delta_i}$ and vice versa. Combining this fact and the above Lemma, we get the following Proposition:
\begin{prop}\label{tangent2}
Let $\widetilde{L}=q_2^{-1}(l)\subset \mathcal{U}_i^{\delta_i^c}$ for $l\in \mathcal{M}_i^{\delta_i^c}$. If $\delta_i$ is not an exposed short root, then for the relative tangent bundle $T_{\mathcal{U}_i^{\delta_i^c}/X}$, we have
$$T_{\mathcal{U}_i^{\delta_i^c}/X}|_{\widetilde{L}}=\mathcal{O}_{\widetilde{L}}(-1)^{N},~N=dim ~\mathcal{U}_i^{\delta_i^c}-dim X.$$
\end{prop}
\begin{proof}
It's obviously from Proposition \ref{tangent}.
\end{proof}

Similarly, by combining with Lemma \ref{bn}, \ref{c1}, \ref{ck}, \ref{tangent3} and Proposition \ref{fk}, we can immediately draw the following proposition.

\begin{prop}\label{tangent4}
	Let $\widetilde{L}=q_2^{-1}(l)\subset \mathcal{U}_i^{\delta_i^c}$ for $l\in \mathcal{M}_i^{\delta_i^c}$. If $\delta_i$ is an exposed short root, then for the relative tangent bundle $T_{\mathcal{U}_i^{\delta_i^c}/X}$,  the splitting type of $T_{\mathcal{U}_i^{\delta_i^c}/X}$ takes several forms: $(-1,\ldots,-1,-2)$, $(-2,\ldots,-2)$, $(-2,-2,-1)$, $(-2,-2,-2,-1,-1,-1)$ or $(-3)$.
\end{prop}

According to Theorem \ref{lines}, \ref{lines2}, if $\delta_i$ is not an exposed short root, then $\mathcal{M}_i^{\delta_i^c}$ is the space of of lines of class $\check{\delta}_i$ and $\mathcal{M}_x^{\delta_i^c}$ is the space of lines of class $\check{\delta}_i$ through $x$.

Using the above propositions, we have the similar results for rational homogeneous spaces.

\begin{thm}\label{longroot2}
Fix $\delta_i\in I_i$ 
and assume that $\delta_i$ is not an exposed short root. Let $E$ be a holomorphic $r$-bundle over $X$ of type $\underline{a}_E^{(\delta_i)}=(a_1^{(\delta_i)},\ldots,a_r^{(\delta_i)}), ~ a_1^{(\delta_i)}\geq\cdots\geq a_r^{(\delta_i)}$ with respect to $\mathcal{M}_i^{\delta_i^c}$. If for some $t<r$,
\[
a_t^{(\delta_i)}-a_{t+1}^{(\delta_i)}\geq
\left\{
\begin{array}{ll}
1, & and~N(\delta_i)~ \text{fits Case I}\\
2, & and ~N(\delta_i)~ \text{fits Case II},
\end{array}
\right.\]
then there is a normal subsheaf $K\subset E$ of rank $t$ with the following properties: over the open set $V_E=q_1(q_2^{-1}(U_E^{(\delta_i)}))\subset X$, where $U_E^{(\delta_i)}$ is an open set in $\mathcal{M}^{\delta_i^c}$, the sheaf $K$ is a subbundle of $E$, which on the line $L\subset X$ given by $l\in U_E^{(\delta_i)}$ has the form
$$K|_L\cong\oplus_{s=1}^{t}\mathcal{O}_L(a_s^{(\delta_i)}).$$
\end{thm}

\begin{proof}
	The Proposition \ref{tangent2} and the Lemma \ref{Descent} play important roles in our proof and our proof applies almost verbatim the proof of Theorem 5.7 given in \cite{D-F-G}.
\end{proof}

\begin{cor}\label{longrc}
With the same assumption as Theorem \ref{longroot2}. For a $\delta_i$-semistable $r$-bundle $E$ over $X$ of type $\underline{a}_E^{(\delta_i)}=(a_1^{(\delta_i)},\ldots,a_r^{(\delta_i)}), a_1^{(\delta_i)}\geq\cdots\geq a_r^{(\delta_i)}$ with respect to $\mathcal{M}_i^{\delta_i^c}$, we have \[
a_s^{(\delta_i)}-a_{s+1}^{(\delta_i)}\leq 1~~ \text{for all}~ s=1,\ldots, r-1.
\]
In particular, if $N(\delta_i)$ fits Case I, then we have $a_s^{(\delta_i)}$'s are constant for all $1\leq s\leq r$.
\end{cor}

When $\alpha_j$ is an exposed short root of $(\mathcal{D},I)$, $\mathcal{M}_i^{\delta_i^c}$ is not the space of of lines of class $\check{\delta}_i$, but only a closed $G$-orbit.

\begin{thm}\label{shortroot2}
Fix $\delta_i\in I_i$ and assume that $\delta_i$ is an exposed short root. Let $E$ be a holomorphic $r$-bundle over $X$ of type $\underline{a}_E^{(\delta_i)}=(a_1^{(\delta_i)},\ldots,a_r^{(\delta_i)}), ~ a_1^{(\delta_i)}\geq\cdots\geq a_r^{(\delta_i)}$ with respect to $\mathcal{M}_i^{\delta_i^c}$. If for some $t<r$,
\[
a_t^{(\delta_i)}-a_{t+1}^{(\delta_i)}\geq
\left\{
\begin{array}{ll}
1, & and~N(\delta_i)~ \text{fits Case I}\\
4, & and ~N(\delta_i)~ \text{fits Case II},
\end{array}
\right.\]
then there is a normal subsheaf $K\subset E$ of rank $t$ with the following properties: over the open set $V_E=q_1(q_2^{-1}(U_E^{(\delta_i)}))\subset X$, where $U_E^{(\delta_i)}$ is an open set in $\mathcal{M}^{\delta_i^c}$, the sheaf $K$ is a subbundle of $E$, which on the line $L\subset X$ given by $l\in U_E^{(\delta_i)}$ has the form
$$K|_L\cong\oplus_{s=1}^{t}\mathcal{O}_L(a_s^{(\delta_i)}).$$
\end{thm}
\begin{proof}
	The Proposition \ref{tangent4} and the Lemma \ref{Descent} play important roles in our proof and our proof applies almost verbatim the proof of Theorem 5.7 given in \cite{D-F-G}.
\end{proof}

\begin{cor}\label{shortrc}
For a $\delta_i$-semistable $r$-bundle $E$ over $X$ of type $\underline{a}_E^{(\delta_i)}=(a_1^{(\delta_i)},\ldots,a_r^{(\delta_i)}), a_1^{(\delta_i)}\geq\cdots\geq a_r^{(\delta_i)}$ with respect to $\mathcal{M}_i^{\delta_i^c}$, we have \[
	a_s^{(\delta_i)}-a_{s+1}^{(\delta_i)}\leq 3~~ \text{for all}~ s=1,\ldots, r-1.
	\]
	In particular, if $N(\delta_i)$ fits Case I, then we have $a_s^{(\delta_i)}$'s are constant for all $1\leq s\leq r$.
\end{cor}

From Proposition \ref{polyuni}, Corollary \ref{longrc} and Corollary \ref{shortrc}, we can have the following result.
\begin{cor}
Let $X=G/B$, where $G$ is a semi-simple Lie group and $B$ is a Borel subgroup of $G$. If an $r$-bundle $E$ is $\delta_i$-semistable for all $i$ and $\delta_i~(1\le i\le m)$ over $X$, then $E$ splits as a direct sum of line bundles.
\end{cor}
\begin{proof}
The assumption tells us that $E$ is poly-uniform with respect to $\mathcal{M}_i^{\delta_i^c}$ for each $i$ and $\delta_i$ and the splitting type is $(a^{(\delta_i)}, \ldots, a^{(\delta_i)})$ by Corollary \ref{longrc} and Corollary \ref{shortrc}. We therefore conclude by Proposition \ref{polyuni}.
\end{proof}

\section*{Acknowledgements}
We would like to show our great appreciation to professor B. Fu, L. Manivel, R. Mu\~{n}oz, G. Occhetta, L. E. Sol\'{a} Conde and K. Watanabe for useful discussions. In particular, we would like to thank professor L. E. Sol\'{a} Conde for recommending us their interesting papers and explaining the details in their papers with patience, especially for the method of calculating the relative tangent bundles between two generalized flag manifolds.
\section{Appendix}
\begin{table}[htbp]
	\caption{$E_6$-type}\label{E6}
	\centering
	\begin{tabular}{|c|c|c|c|c|}
		\hline
		$G$&node $k$ & $H/P_{N(k)}$ & tag&weights \\
		\hline
		$E_6$&$1$ & $D_5 /P_5$ & $(00001)$ &\tabincell{c}{$\{(12211),(11211),(01211),(11111),(11101),$\\$(01111),(01101),(00111),(00101),(00001)\}$}\\
		\hline
		$E_6$&$2$ & $A_5/P_3$ & $(00100)$ &\tabincell{c}{$\{(11100),(11110),(11111),(01100),(01110),$ \\ $(01111),(00100),(00110),(00111)\}$}\\
		\hline
		$E_6$&$3$ & $A_1/P_1\times A_4/P_2$ & $(1,0100)$ & \tabincell{c}{$\{(1)\}\times\{(1100),(1110),(1111),$\\$(0100),(0110),(0111)\}$}\\
		\hline
		$E_6$&$4$ & $A_2/P_2\times A_1/P_1\times A_2/P_1$ & $(01,1,10)$ & $\{(01),(11)\}\times\{(1)\}\times\{(10),(11)\}$\\
		\hline
		$E_6$&$5$ & $A_4/P_3\times A_1/P_1$ & $(0010,1)$ &  \tabincell{c}{$\{(1110),(1111),(0110),$\\$(0111),(0010),(0011)\}\times\{(1)\}$}\\
		\hline
		$E_6$&$6$ & $D_5/P_5$ & $(00001)$ &\tabincell{c}{ $\{(12211),(11211),(01211),(11111),(11101),$\\$(01111),(01101),(00111),(00101),(00001)\}$}\\
		\hline
	\end{tabular}
\end{table}

\begin{table}[htbp]
	\caption{$E_7$-type}\label{E7}
	\centering
	\begin{tabular}{|c|c|c|c|c|}
		\hline
		$G$&node $k$ & $H/P_{N(k)}$ & tag&weights \\
		\hline
		$E_7$&$1$ & $D_6/P_6$ & $(000001)$ &\tabincell{c}{$\{(122211),(112211),(012211),(111111),$\\$ (111101),(011111),(001101),(000111),$\\$ (000101),(000001),(111211),(011211),$\\$ (011101),(001211),(001111)\}$}\\
		\hline
		$E_7$&$2$ & $A_6/P_3$ & $(001000)$ &\tabincell{c}{$\{(111000),(111100),(111110),(111111),$ \\ $(011000),(011100),(011110),(011111),$ \\ $(001000),(001100),(001110),(001111)\}$}\\
		\hline
		$E_7$&$3$ & $A_1/P_1\times A_5/P_2$ & $(1,01000)$ & \tabincell{c}{$\{(1)\}\times\{(11000),(11100),$\\$(11110),(11111),(01000),$\\$ (01100),(01110),(01111)\}$}\\
		\hline
		$E_7$&$4$ & $A_2/P_2\times A_1/P_1\times A_3/P_1$ & $(01,1,100)$ & $\{(01),(11)\}\times\{(1)\}\times\{(100),(110),(111)\}$\\
		\hline
		$E_7$&$5$ & $A_4/P_3\times A_2/P_1$ & $(0010,10)$ &  \tabincell{c}{$\{(1110),(1111),(0110),(0111)$\\$(0010),(0011)\}\times\{(10),(11)\}$}\\
		\hline
		$E_7$&$6$ & $D_5/P_5\times A_1/P_1$ & $(00001,1)$ &\tabincell{c}{ $\{(12211),(11211),(01211),(11111),$\\$(11101),(01111),(01101),(00111),$\\$ (00101),(00001)\}\times\{(1)\}$}\\
		\hline
		$E_7$&$7$ & $E_6/P_6$ & $(000001)$ &\tabincell{c}{ $\{(000001),(000011),(000111),(010111),$\\$(001111),(101111),(011111),(111111)$\\$(011211),(111211),(011221),(112211),$\\$(111211),(112221),(112321),(122321)\}$}\\
		\hline
	\end{tabular}
\end{table}

\begin{table}[htbp]
	\caption{$E_8$-type}\label{E8}
	\centering
	\begin{tabular}{|c|c|c|c|c|}
		\hline
		$G$&node $k$ & $H/P_{N(k)}$ & tag&weights \\
		\hline
		$E_8$&$1$ & $D_7/P_7$ & $(0000001)$ &\tabincell{c}{$\{(1222211),(1122211),(1112211),$\\$(1111211), (1111111),(1111101),$\\$ (0122211),(0112211), (0111211),$\\$ (0111111),(0111101),(0012211),$\\$ (0011211), (0011111),(0011101),$\\$ (0001211),(0001111), (0001101),$\\$ (0000111),(0000101),(0000001)\}$}\\
		\hline
		$E_8$&$2$ & $A_7/P_3$ & $(0010000)$ &\tabincell{c}{$\{(1110000),(1111000),(1111110),$\\$(1111110),(1111111) ,(0110000),$\\$ (0111000), (0111100),(0111110),$\\$ (0111111),(0010000),(0011000),$ \\ $ (0011100),(0011110),(0011111)\}$}\\
		\hline
		$E_8$&$3$ & $A_1/P_1\times A_6/P_2$ & $(1,010000)$ & \tabincell{c}{$\{(1)\}\times\{(110000),(111000),$\\$ (111100),(111110),(111111),$\\$ (010000),(011000), (011100),$\\$ (011110),(011111)\}$}\\
		\hline
		$E_8$&$4$ & $A_2/P_2\times A_1/P_1\times A_4/P_1$ & $(01,1,1000)$ & \tabincell{c}{$\{(01),(11)\}\times\{(1)\}\times\{(1000),$\\$(1100),(1110),(1111)\}$}\\
		\hline
		$E_8$&$5$ & $A_4/P_3\times A_3/P_1$ & $(0010,100)$ &  \tabincell{c}{$\{(1110),(1111),(0110),(0111),$\\$(0010),(0011)\}\times\{(100),(110),(111)\}$}\\
		\hline
		$E_8$&$6$ & $D_5/P_5\times A_2/P_1$ & $(00001,10)$ &\tabincell{c}{ $\{(12211),(11211),(01211),$\\$(11111),(11101),(01111),$\\$(01101),(00111), (00101),$\\$ (00001)\}\times\{(10),(11)\}$}\\
		\hline
		$E_8$&$7$ & $E_6/P_6\times A_1/P_1$ & $(000001,1)$ &\tabincell{c}{ $\{(000001),(000011),(000111),$\\$ (010111),(001111),(101111),$\\$ (011111),(111111),(011211),$\\$ (111211),(011221),(112211),$\\$(111211),(112221),(112321),$\\$ (122321)\}\times\{(1)\}$}\\
		\hline
		$E_8$&$8$ & $E_7/P_7$ & $(0000001)$ &\tabincell{c}{ $\{(000001),(000011),(000111),$\\$(0101111),(0001111),(0112221),$\\$ (0112211), (0112111)(0111111),$\\$ (0011111),(1112221),(1112211),$\\$ (1112111),(1111111),(1011111),$\\$ (1224321),(1223321),(1123321),$\\$ (1223221), (1123221),(1122221),$\\$ (1223211),(1123211),(1122211),$\\$ (1122111),(1234321),(2234321)\}$}\\	
		\hline
	\end{tabular}
\end{table}

\begin{table}[htbp]
	\caption{$F_4,G_2$-type}\label{FG}
	\centering
	\begin{tabular}{|c|c|c|c|c|}
		\hline
		$G$&node $k$ & $H/P_{N(k)}$ & tag&weights \\
		\hline
		$F_4$&$1$ & $C_3/P_3$ & $(001)$ &$\{(121),(111),(011),(221),(021),(001)\}$\\
		$F_4$&$2$ & $A_1/P_1\times A_2/P_1$ & $(1,10)$ & $\{(1)\}\times\{(10),(11)\}$\\
		$F_4$&$3$ & $A_2/P_2\times A_1/P_1$ & $(02,1)$ & $\{(01),(11)\}\times\{(1)\}$\\
		$F_4$&$4$ & $B_3/P_3$ & $(001)$ & $\{(122),(112),(012),(111),(011),(001)\}$\\
		\hline
		$G_2$&$1$ & $A_1/P_1$ & $(3)$ &$\{(1)\}$\\
		$G_2$&$2$ & $A_1/P_1$ & $(1)$ & $\{(1)\}$\\
		\hline
	\end{tabular}
\end{table}

\bibliography{ref}

\begin{thebibliography}{10}

\bibitem{A-W}
M.~Andreatta and J.~Wi{\'s}niewski.
\newblock On manifolds whose tangent bundle contains an ample locally free
  subsheaf.
\newblock {\em Invent. Math.}, 146:209--217, 2001.

\bibitem{Bal2}
E.~Ballico.
\newblock Uniform vector bundles on quadrics.
\newblock {\em Ann. Univ. Ferrara Sez. VII (N.S.)}, 27(1):135--146, 1982.

\bibitem{Bal}
E.~Ballico.
\newblock Uniform vector bundles of rank $(n+ 1) $ on $ \mathbb{P}^{n} $.
\newblock {\em Tsukuba J. Math.}, 7(2):215--226, 1983.

\bibitem{B-R}
A.~Borel and R.~Remmert.
\newblock \"{U}ber kompakte homogene k\"{a}hlersche mannigfaltigkeiten.
\newblock {\em Math. Ann.}, 145:429--439, 1961/1962.

\bibitem{R-C}
R.~Carter.
\newblock {\em Lie Algebras of Finite and Affine Type}.
\newblock Cambridge studies in advanced mathematics, 2005.

\bibitem{D-F-G}
R.~Du, X.~Fang, and Y.~Gao.
\newblock Vector bundles on flag varieties.
\newblock {\em arXiv:1905.10151}.

\bibitem{HD}
H.~Duan and X.~Zhao.
\newblock The chow ring of generalized grassmannians.
\newblock {\em Foundations of Computational Mathematics}, 10(3):245--274, 2005.

\bibitem{Da}
D.~Eisenbud and J.~Harris.
\newblock {\em 3264 All that-Interesection Theory in Algebraic Geometry}.
\newblock Cambridge University Press, 2011.

\bibitem{Ele}
G.~Elencwajg.
\newblock Les fibr\'es uniformes de rang 3 sur $\mathbb{P}^2(\mathbb{C})$ sont
  homog\'enes.
\newblock {\em Math.Ann}, 231:217--227, 1978.

\bibitem{E-H-S}
G.~Elencwajg, A.~Hirschowitz, and M.~Schneider.
\newblock Les fibr\'es uniformes de rang n sur $\mathbb{P}^n(\mathbb{C})$ sont
  ceux qu'on croit.
\newblock {\em Progr. Math.}, 7:37--63, 1980.

\bibitem{Ell}
P.~Ellia.
\newblock Sur les fibr{\'e}s uniformes de rang $(n+1)$ sur $\mathbb{P}^{n} $.
\newblock {\em M{\'e}m. Soc. Math. France (N.S.)}, 7:1--60, 1982.

\bibitem{Guy}
M.~Guyot.
\newblock Caract{\'e}risation par l'uniformit{\'e} des fibr{\'e}s universels
  sur la grassmanienne.
\newblock {\em Math. Ann.}, 270(1):47--62, 1985.

\bibitem{H-L}
D.~Huybrechts and M.~Lehn.
\newblock {\em The geometry of moduli spaces of sheaves, second ed.}
\newblock Cambridge University Press, 2010.
\newblock xviii+325 pp.

\bibitem{Hwa}
J-M. Hwang.
\newblock Geometry of minimal rational curves on fano manifolds.
\newblock {\em School on Vanishing Theorems and Effective Results in Algebraic
  Geometry (Trieste, 2000), ICTP Lect. Notes, vol. 6, Abdus Salam Int. Cent.
  Theoret. Phys., Trieste, 2001}, pages 335--393, 2001.

\bibitem{K-S}
Y.~Kachi and E.~Sato.
\newblock Segre’s reflexivity and an inductive characterization of
  hyperquadrics.
\newblock {\em Mem. Amer. Math. Soc.}, 160(763), 2002.

\bibitem{L-M}
J.M. Landsberg and L.~Manivel.
\newblock On the projective geometry of rational homogeneous varieties.
\newblock {\em Comment. Math. Helv.}, 78:65--100, 2003.

\bibitem{M-O-C}
R.~Mu{\~{n}}oz, G.~Occhetta, and L.~E. {Sol{\'{a}} Conde}.
\newblock On uniform flag bundles on fano manifolds.
\newblock {\em arXiv:1610.05930}.

\bibitem{M-O-C2}
R.~Mu{\~{n}}oz, G.~Occhetta, and L.~E. {Sol{\'a} Conde}.
\newblock Uniform vector bundles on fano manifolds and applications.
\newblock {\em J. Reine Angew. Math. (Crelles Journal)}, 664:141--162, 2012.

\bibitem{M-O-C3}
R.~Mu{\~{n}}oz, G.~Occhetta, and L.~E. {Sol{\'a} Conde}.
\newblock Splitting conjectures for uniform flag bundles.
\newblock {\em European Journal of Mathematics}, 6:430–452, 2020.

\bibitem{M-O-C-W}
R.~Mu{\~{n}}oz, G.~Occhetta, L.~E. {Sol{\'{a}} Conde}, and K.~Watanabe.
\newblock Rational curves,dynkin diagrams and fano manifolds with nef tangent
  bundle.
\newblock {\em Math.Ann.}, 361:583--609, 2015.

\bibitem{O-C-W-W}
G.~Occhetta, L.~E. {Sol{\'{a}} Conde}, K.~Watanabe, and J.~A. Wi{\'{s}}niewski.
\newblock Fano manifolds whose elementary contractions are smooth
  $\mathbb{P}^1$-fibrations: a geometric characterization of flag varieties.
\newblock {\em Ann. Sc. Norm. Super. Pica. Cl. Sci.}

\bibitem{O-C-W}
G.~Occhetta, L.~E. {Sol{\'{a}} Conde}, and J.~A. Wi{\'{s}}niewski.
\newblock Flag bundles on fano manifolds.
\newblock {\em Journal de Mathmatiques Pures et Appliques}, 106(4):651--669,
  2016.

\bibitem{O-S-S}
C.~Okonek, M.~Schneider, and H.~Spindler.
\newblock {\em Vector bundles on complex projective spaces}.
\newblock Birkh$\ddot{a}$user/Springer Basel AG, Basel, 2011.
\newblock viii+239 pp.

\bibitem{Sat}
E.~Sato.
\newblock Uniform vector bundles on a projective space.
\newblock {\em J. Math. Soc. Japan}, 28(1):123--132, 1976.

\bibitem{Sch}
R.~L.~E. Schwarzenberger.
\newblock Vector bundles on the projective plane.
\newblock {\em Proc. London Math. Soc.}, 11(3):623--640, 1961.

\bibitem{Ha}
H.~Tamvakis.
\newblock Quantum cohomology of isotropic grassmannians.
\newblock {\em Geometric Methods in Algebra and Number Theory}, 235:311--338,
  2006.

\bibitem{Ven}
A.~{Van de Ven}.
\newblock On uniform vector bundles.
\newblock {\em Math. Ann.}, 195:245--248, 1972.

\end{thebibliography}

\end{document}